\documentclass{amsart}
\usepackage{hyperref}

\usepackage{enumerate}

\newcommand{\arxiv}[1]{\href{http://arxiv.org/abs/#1}{arXiv:#1}}
\newcommand*{\mailto}[1]{\href{mailto:#1}{\nolinkurl{#1}}}

\newtheorem{theorem}{Theorem}[section]
\newtheorem{lemma}[theorem]{Lemma}
\newtheorem{corollary}[theorem]{Corollary}
\newtheorem{remark}[theorem]{Remark}

\newtheorem{definition}[theorem]{Definition}
\newcommand{\R}{\mathbb{R}}

\newcommand{\C}{\mathbb{C}}

\newcommand{\nn}{\nonumber}
\newcommand{\beq}{\begin{equation}}
\newcommand{\eeq}{\end{equation}}
\newcommand{\bea}{\begin{eqnarray}}
\newcommand{\eea}{\end{eqnarray}}
\newcommand{\ul}{\underline}
\newcommand{\ol}{\overline}
\newcommand{\pa}{\partial}

\newcommand{\norm}[1]{\lVert#1 \rVert}
\newcommand{\abs}[1]{\lvert#1 \rvert}

\newcommand{\I}{\mathrm{i}}
\newcommand{\E}{\mathrm{e}}
\newcommand{\clos}{\mathop{\mathrm{clos}}}

\newcommand{\im}{\mathop{\mathrm{Im}}}
\newcommand{\weak}{\mathop{\mathrm{weak}}}

\DeclareMathOperator{\res}{Res}
\DeclareMathOperator{\loc}{loc}

\def\XXint#1#2#3{{\setbox0=\hbox{$#1{#2#3}{\int}$}
     \vcenter{\hbox{$#2#3$}}\kern-.5\wd0}}

\newcommand{\si}{\sigma}
\newcommand{\la}{\lambda}
\newcommand{\ga}{\gamma}

\numberwithin{equation}{section}

 \newcommand{\noprint}[1]{}

\begin{document}

\title[Inverse scattering theory for steplike potetnials]{Inverse scattering theory for Schr\"odinger operators
with steplike potentials}

\author[I. Egorova]{Iryna Egorova}
\address{B. Verkin Institute for Low Temperature Physics\\ 47, Lenin ave\\ 61103 Kharkiv\\ Ukraine\\ and Faculty of Mathematics\\ University of Vienna\\
Oskar-Morgenstern-Platz 1\\ 1090 Wien\\ Austria }
\email{\href{mailto:iraegorova@gmail.com}{iraegorova@gmail.com}}

\author[Z.\ Gladka]{Zoya Gladka}
\address{B. Verkin Institute for Low Temperature Physics\\ 47, Lenin ave\\ 61103 Kharkiv\\ Ukraine}
\email{\href{mailto:gladkazoya@gmail.com}{gladkazoya@gmail.com}}

\author[T.L.\ Lange]{Till Luc Lange}
\address{Faculty of Mathematics\\ University of Vienna\\
Oskar-Morgenstern-Platz 1\\ 1090 Wien\\ Austria}
\email{\mailto{till.luc.lange@univie.ac.at}}

\author[G.\ Teschl]{Gerald Teschl}
\address{Faculty of Mathematics\\ University of Vienna\\
Oskar-Morgenstern-Platz 1\\ 1090 Wien\\ Austria\\ and International Erwin Schr\"odinger
Institute for Mathematical Physics\\ Boltzmanngasse 9\\ 1090 Wien\\ Austria}
\email{\mailto{Gerald.Teschl@univie.ac.at}}
\urladdr{\url{http://www.mat.univie.ac.at/~gerald/}}

\thanks{Zh. Mat. Fiz. Anal. Geom. {\bf 11}, 123--158 (2015)}
\thanks{{\it Research supported by the Austrian Science Fund (FWF) 
under Grants No.\ Y330, V120 and W1245}.}

\keywords{Schr\"odinger operator, inverse scattering theory, steplike potential}
\subjclass[2000]{Primary 34L25, 81U40; Secondary 34B30, 34L40}

\begin{abstract}
We study the direct and inverse scattering problem for the one-dimensional Schr\"odinger equation with steplike
potentials. We give necessary and sufficient conditions for the scattering data to correspond to a
potential with prescribed smoothness and prescribed decay to their asymptotics. These results are important for solving the Korteweg--de Vries equation via
the inverse scattering transform.
\end{abstract}

\maketitle


\section{Introduction}

Among various direct/inverse spectral problems the scattering problem on the whole axis for one-dimensional Schr\"odinger
operators with decaying potentials takes a particular place as  one of the most rigorously investigated spectral
problems. Being considered first by Kay and Moses \cite{KM56} on a physical level of rigor,
it was rigorously studied by Faddeev \cite{Fa58}, and then revisited independently by Marchenko \cite{M} and by Deift and Trubowitz \cite{DT79}. In particular, Faddeev \cite{Fa58} considered the inverse problem in the class of potentials which have
a finite first moment (i.e., \eqref{decay} below with $c_-=c_+=0$ and $m=1$) but the importance of the behavior of the
scattering coefficients at the bottom of the continuous spectrum was missed. A complete solution was  given by
Marchenko \cite{M} (see also Levitan \cite{L}) for the first moment ($m=1$) and by Deift and Trubowitz \cite{DT79} for the second moment ($m=2$) who also gave an example
showing that some condition on the aforementioned behavior is necessary for solving the inverse problem.

The next simplest case is  the so-called steplike case where the potential tends to different constants one the left and right
half-axes. The corresponding scattering problem was first considered on an informal level by Buslaev and Fomin in \cite{BF} who
studied mostly the direct scattering problem and derived the main equation of the inverse problem --- the Gelfand--Levitan--Marchenko (GLM) equation.
A complete solution of the direct and inverse scattering problem for steplike potentials with a finite second moment (i.e., \eqref{decay} below with $m=2$) was solved rigorously by Cohen and Kappeler \cite{CK} (see also \cite{DS} and \cite{GNP}). While several aspects in the
steplike case are similar to the decaying case, there are also some distinctive differences due to the presence of spectrum
of multiplicity one. Moreover, there have also been further generalizations to the case of periodic backgrounds made by Firsova \cite{F1,F2,F3}
and to steplike finite-gap backgrounds by Boutet de Monvel and two of us \cite{bet} (see also \cite{MLT}) and to steplike almost periodic backgrounds by
Grunert \cite{Gr1,Gr2}. We refer to these publications for more information.

Our aim in the present paper is to use  Marchenko's approach for the generalization of the results of \cite{CK}
to the case of steplike potentials with finite first moment which  turns out to be much more delicate than the
second moment. Note that this question is partly studied in \cite{Ba}. In fact, we will also give a complete solution of the inverse problem for potentials with any
given number of moments $m\ge 1$ and any given number of derivatives $n\ge 0$ which has important
applications for the solution of the Korteweg--de Vires (KdV) equation.

As is well known, the inverse scattering transform (IST) is the main ingredient for solving and understanding the
solutions of the KdV (as well as the associated modified KdV) equation. In fact, applications of the IST to the initial value problem for KdV 
were already considered by many authors (see for example the monographs \cite{evh}, \cite{M}, \cite{Zakh}). For the steplike case
this was first done  by Cohen \cite{C} and Kappeler \cite{Kap}. For more general backgrounds we refer to \cite{F4}
and to the more recent works \cite{EGT}, \cite{ET1}, \cite{ET3} as well as the references therein.
For the long-time asymptotics of solutions, we refer to  \cite{Zakh}, \cite{Kh}, \cite{Ven}, \cite{Bik1}, \cite{Bik2} and
to \cite{AB}, \cite{GT}, \cite{MLT2}, \cite{KM}, \cite{EGKT},  \cite{LN}, \cite{LN1} for more recent developments.
In a forthcoming paper \cite{EGLT2} we will apply the inverse scattering transform to solve the Cauchy problem
for the Korteweg--de Vries equation for initial conditions in the class of potentials investigated in the present
paper, extending the results from \cite{ET3}.

We consider the spectral problem
\beq\label{Sp}
(L f)(x):=- \frac{d^2 }{dx^2}f(x) + q(x)f(x)=\lambda f(x),\qquad x\in\R,
\eeq
with a steplike potential $q(x)$  such that
\[
q(x)\to c_{\pm}, \ \ \mbox{as}\ \ x\to \pm\infty,
\]
where $c_+, c_-\in\R$ are in general different values. 
Everywhere in this paper we assume that  $q\in L^1_{\loc}(\R)$ and tends to its background asymptotics $c_+$
and $c_-$ with $m$ "moments" finite:
\beq\label{decay}
\int_0^{+\infty} (1+|x|^m)(|q(x)-c_+| + |q(-x)-c_-|)dx<\infty,
\eeq
 where $m\geq 0$ is a fixed integer. 
 \begin{definition} Let $m\geq 0$ and $n\geq 0$ be  integers and $f:\R \to \R$ be an $n$ times differentiable function. We say that $f\in \mathcal L_m^n(\R_\pm)$ if $f^{(j)}(x) (1+|x|^m) \in L^1(\R_\pm)$ for $j=0,1,\dots,n$.
\end{definition}

Note, that $f\in  \mathcal L_m^0(\R_\pm)$ means that $\int_{\R_\pm}|f(x)|(1+|x|^m)dx<\infty$.
By this definition $\mathcal L_0^0(\R_\pm)=L^1(\R_\pm)\cap L^1_{\loc}(\R)$ and $\mathcal L_0^{j}(\R_\pm)=\{f:\  f^{(i)}\in \mathcal L_0^0(\R_\pm),\  0\leq i\leq j\}$.

\begin{definition}
 Let $c_\pm$ be given real values and let $m\geq 0$, $n\geq 0$ be given integers. We say that $q\in\mathcal L_m^n(c_+, c_-)$ if $q_\pm(\cdot):=q(\cdot)-c_\pm\in \mathcal L_m^n(\R_\pm)$.
\end{definition}
Note that  $q\in\mathcal L_m^0(c_+, c_-)$ if condition \eqref{decay} holds. If $q\in \mathcal L_m^n(c_+, c_-)$ with $n\geq 1$ then in addition
\beq\label{decay1}
\int_{\R}(1+|x|^m)|q^{(i)}(x)|dx<\infty, \quad i=1,\dots,n.
\eeq
The aim of this paper is a complete study of the direct and inverse scattering problem for potentials from the classes $\mathcal L^n_m(c_+,c_-)$ with $m\geq 1$.
In particular, we propose necessary and sufficient conditions on the set of scattering data associated with such potentials.
The following notations will be used throughout this paper:

Abbreviate 
\beq\label{nott}\ul c=\min\{c_-, c_+\}, \qquad \ol c=\max\{c_-, c_+\},
\eeq and $\mathcal D:=\C\setminus\Sigma$, where  
$\Sigma=\Sigma^u\cup\Sigma^l$ with $\Sigma^u=\{\la^u=\la+\I 0, \la\in[\ul c,\infty)\}$ and  $\Sigma^l=\{\la^l=\la-\I 0, \la\in[\ul c,\infty)\}$. We treat the boundary of the domain $\mathcal D$ as consisting of two sides of cuts along the interval $ [\ul c,\infty)$, with distinguished  points $\la^u$ and $\lambda^l$ on this boundary. In equation \eqref{Sp}  the spectral parameter $\la$ belongs to the set $\clos(\mathcal D)$, where $\clos(\mathcal D)=\mathcal D\cup\Sigma^u\cup\Sigma^l$. 
Along   with  $\la$ we will use two more spectral parameters
\beq\label{defkpm}
k_\pm:=\sqrt{\la - c_\pm},
\eeq
which map the domains $\C\setminus [c_\pm, \infty)$ conformally onto $\C^+$. Thus there is a one to one correspondence between the parameters $k_\pm$ and $\la$.

\section{The Direct scattering problem}

\subsection{Properties of the Jost solutions}

In this subsection we collect some well-known properties of the Jost solutions for \eqref{Sp}
with $q\in \mathcal L_1^0(c_+,c_-)$ and establish additional properties of these solutions for a potential from the class
$\mathcal L_m^n(c_+,c_-)$ with $m\geq 2$ or $n\geq 1$.
All the estimates below are one-sided and hence are generated by the behavior of the potential on one half axis.
For $q_{\pm}(\cdot)=q(\cdot)-c_\pm\in  \mathcal L_m^n(\R_\pm)$, $m\geq 1$, $n\geq 0$, introduce
nonnegative, as $x\to\pm \infty$ nonincreasing functions
\beq\label{defsig}
\sigma_{\pm,i}(x):=\pm\int_x^{\pm\infty} |q_{\pm}^{(i)}(\xi)|d\xi, \quad \hat\sigma_{\pm,i}(x):= \pm\int_x^{\pm\infty}\sigma_{\pm,i}(\xi) d\xi, \quad i=0,1,\dots,n.
\eeq
Evidently,
\beq\label{decay6}\sigma_{\pm,i}(\cdot)\in  \mathcal L_{m-1}^1(\R_{\pm}),\ m\geq 1, \ 
 \hat \sigma_{\pm,i}(\cdot)\in   \mathcal L_{m-2}^2(\R_{\pm}),\ m\geq 2,
\eeq
\beq \label{decay5}\hat\sigma_{\pm,i} (x)\downarrow 0\ \ \mbox{as}\ \ x\to \pm\infty,\ \ \mbox{for}\ q_\pm\in \mathcal L_1^n(\R_\pm),\ \ i=0,1,\dots,n.
\eeq

\begin{lemma} (\cite[Lemmas 3.1.1--3.1.3]{M}).
Let $q_{\pm}(\cdot)=q(\cdot)-c_\pm\in  \mathcal L_1^0(\R_\pm)$. Then for all $\la\in\clos(\mathcal D)$ equation \eqref{Sp} has a solution
$\phi_{\pm}(\la,x)$ which can be represented as
\beq\label{Jost1}
\phi_{\pm}(\la,x)=\E^{\pm\I k_{\pm} x} \pm \int_x^{\pm\infty} K_{\pm}(x,y)\E^{\pm\I k_{\pm} y}dy,
\eeq
where the kernel $K_{\pm}(x,y)$ is real-valued and satisfies the inequality
\beq\label{K1}
|K_{\pm}(x,y)|\leq\frac{1}{2}\sigma_{\pm,0}
\left(\frac{x+y}{2}\right)\exp\left\{\hat\sigma_{\pm,0}(x) -
\hat\sigma_{\pm,0}\left(\frac{x+y}{2}\right)\right\}.
\eeq
Moreover,
\[
K_{\pm}(x,x)=\pm\dfrac{1}{2}\int_x^{\pm\infty}q_{\pm}(\xi)d\xi.
\]
The function $K_{\pm}(x,y)$ has first order partial derivatives which satisfy the inequality
\beq\label{K2}
\left|\frac{\partial K_{\pm}(x_1,x_2)}{\partial x_j} \pm\frac{1}{4}q_{\pm}\left(\frac{x_1 + x_2}{2}\right)\right|\leq
 \eeq
 \[\leq\frac{1}{2}\sigma_{\pm,0}
\left(x\right)\sigma_{\pm,0}
\left(\frac{x_1+x_2}{2}\right)\exp\left\{\hat\sigma_{\pm,0}(x_1) - \hat\sigma_{\pm,0}\left(\frac{x_1+x_2}{2}\right)\right\}.\]
The solution $\phi_{\pm}(\la,x)$ is an analytic function of $k_\pm$ in $\C^+$ and is
continuous up to  $\R$. For all $\la\in\clos(\mathcal D)$ the following estimate is valid
\beq\label{zeros}
\left|\phi_\pm(\la,x) - \E^{\pm\I k_\pm x}\right|\leq\left(\hat\sigma_{\pm,0}(x) -
\hat\sigma_{\pm,0}\left(x \pm\frac{1}{|k_\pm|}\right)\right)\E^{- \im(k_\pm) x +\hat\sigma_{\pm,0}(x)}.
\eeq
For $k_\pm\in\R\setminus\{0\}$ the functions $\phi_\pm(\la,x)$ and $\overline{\phi_\pm(\la,x)}$  are linearly independent
with \[ W(\phi_\pm(\la,\cdot), \overline{\phi_\pm(\la,\cdot)})=\mp 2\I k_\pm,\]
where $W(f,g)=f g^\prime - gf^\prime$ denotes the usual Wronski determinant.
\end{lemma}

Formulas \eqref{K1} and \eqref{K2} together with \eqref{Jost1} and \eqref{decay6}  imply
\begin{corollary}
Let $q_\pm\in \mathcal L^0_m(\R_\pm)$, $m\geq 1$. Then
\beq\label{refff}  K_\pm(x,\cdot),\ \  \frac{\pa K_\pm(x,\cdot)}{\pa x}\in \mathcal L^0_{m-1}(\R_\pm),\ \ m\geq 1,\eeq
and the function $\phi_\pm(\la,x)$
is $m-1$ times differentiable with respect to $k_\pm\in\R$.
\end{corollary}

Note, that the key ingredient for proving the estimates \eqref{K1} and \eqref{K2} is a rigorous investigation of the following integral equation (formula (3.1.12) of \cite{M})
\beq\label{Kint}
K_\pm(x,y)=\pm\frac{1}{2}\int_{\frac{x+y}{2}}^{\pm\infty}q_\pm(\xi)d\xi +
\int_{\frac{x+y}{2}}^{\pm\infty}
d\alpha\int_0^{\frac{y-x}{2}}q_\pm(\alpha - \beta)K_\pm(\alpha - \beta, \alpha + \beta)d\beta.
\eeq 
To further study the properties of the Jost solution we represent \eqref{Jost1} in the form proposed in \cite{DT79}:
\beq\label{Jost11}
\phi_\pm(\la,x)=\E^{\I k_\pm x}\left(1 \pm\int_0^{\pm\infty}B_\pm(x,y)\E^{\pm 2\I k_\pm y}dy\right),
\eeq
where
\beq\label{B-K}
B_\pm(x,y)=2K_\pm(x,x+2y),\ \ B_\pm(x,0)=\pm\int_x^{\pm\infty} q_\pm(\xi)d\xi,
\eeq
and equation \eqref{Kint} transforms into the following integral equation with respect to  $\pm y\geq 0$
\beq \label{Bint}
B_\pm(x,y)=\pm\int_{x+y}^{\pm\infty}q_\pm(s)ds + \int_{x+y}^{\pm\infty}d\alpha
\int_0^y d\beta q_\pm(\alpha - \beta)B_\pm(\alpha - \beta, \beta).\eeq

Equation \eqref{Bint} is the basis for proving the following
\begin{lemma}\label{Best}
Let $n\geq 1$ and $m\geq 1$ be fixed natural numbers and let $q_\pm\in \mathcal L_m^n(\R_\pm)$.
Then the functions $B_\pm( x,y)$ have $n+1$ partial derivatives and the following estimates are valid for $l\leq s\leq n+1$
\beq\label{Best1}
\left|\frac{\partial^{s}}{\partial x^l\,\partial y^{s-l}}B_\pm(x,y) \pm q_\pm^{(s-1)}(x+y)\right|\leq C_{\pm}(x)\nu_{\pm,s}(x)\nu_{\pm,s}(x+y),
\eeq
where
\beq\label{defsig1}
\nu_{\pm,l}(x)=\sum_{i=0}^{l-2} \left(\sigma_{\pm, i}(x) + |q_\pm^{(i)}(x)|\right),\quad l\geq 2,\quad \nu_{\pm,1}(x):=\sigma_{\pm,0}(x),
\eeq
and $C_{\pm}(x)=C_{\pm}(x,n)\in\mathcal C(\R)$ are  positive functions which are nonincreasing as $x\to\pm\infty$.
\end{lemma}

\begin{proof}
Differentiating equation \eqref{Bint} with respect to each variable, we get
\beq\label{der1}
\frac{\partial B_\pm(x,y)}{\partial x}=\mp q_\pm(x+y) -\int_x^{x+y}q_\pm(s)B_\pm(s, x+y-s)ds;\eeq
\beq\label{der2}
\frac{\partial B_\pm(x,y)}{\partial y}=\mp q_\pm(x+y) -\int_x^{x+y}q_\pm(s)B_\pm(s, x+y-s)ds +\int_x^{\pm\infty}q_\pm(\alpha) B_\pm(\alpha,y) d\alpha.
\eeq
From these formulas and \eqref{B-K}, we obtain
\[
\frac{\partial B_\pm(x,0)}{\partial x}=\mp q_\pm(x);\quad
\frac{\partial B_\pm(x,y)}{\partial y}|_{y=0}=\mp q_\pm(x) \pm \frac{1}{2}\left(\int_x^{\pm\infty}q_\pm(\alpha)d\alpha\right)^2,
\]
\beq\label{der4}
\frac{\partial B_\pm(x,y)}{\partial y}=\frac{\partial B_\pm(x,y)}{\partial x} +\int_x^{\pm\infty}q_\pm(\alpha) B_\pm(\alpha,y) d\alpha.
\eeq
We observe that the partial derivatives of $B_\pm$, which contain at least one differentiation with respect to $x$, have the structure
\begin{align}\label{der5}
& \frac{\partial^{p}}{\partial x^k\partial y^{p-k}}B_\pm(x,y)=\mp q_\pm^{(p-1)}(x+y) +D_{\pm,p,k} (x,y) +\\ \nn
& \qquad {} +\int^x_{x+y}q_\pm(\xi)\frac{\partial^{p-1}}{\partial y^{p-1}}
B_\pm(\xi,x+y-\xi)d\xi,\quad p> k\geq 1,
\end{align}
where $D_{\pm,p,k}(x,y)$ is the sum of all derivatives of all integrated terms which appeared after $p-1$ differentiation of the upper and lower limits of the integral on the right hand side of \eqref{der1}. The integrand in \eqref{der5} at the lower limit of integration has value
\[
q_\pm(\xi)\frac{\partial^{p-1}}{\partial y^{p-1}} B_\pm(\xi, x+y-\xi)|_{\xi=x+y}=q_\pm(x+y)B_{\pm, p-1}(x+y),
\]
where
\beq\label{der6}
B_{\pm,r}(\xi)= \frac{\partial^{r}}{\partial t^{r}}B_\pm(\xi,t)|_{t=0}.
\eeq
Thus, further derivatives of such a term do not depend on whether we differentiate it with respect to $x$ or $y$.
The same integrand at the upper limit has the value
$q_\pm(x)\frac{\partial^{r-1}}{\partial y^{r-1}}B_\pm(x,y)$, and it will appear only after a differentation with respect to $x$.
Taking all this into account, we conclude that $D_{\pm,p,k}(x,y)$ in \eqref{der5} can be represented as
\[
D_{\pm,p,k}(x,y)=(1-\delta(k,1))\frac{\partial^{p-k}}
{\partial y^{p-k}}\sum_{s=2}^{ k}\frac{\partial^{k-s}}{\partial x^{k-s}}\left(
q_\pm(x)\frac{\partial^{s-2}}{\partial y^{s-2}}B_\pm(x,y)\right)-D_{\pm,p}(x+y),
\]
where
$\delta(r,s)$ is the Kronecker delta (i.e. the first summand is absent for $k=1$) and
\beq\label{der9}
D_{\pm,p}(\xi):=\sum_{s=0}^{p-2}\frac{d^{p-s}}
{d \xi^{p-s} }\left(q_\pm(\xi)B_{\pm,s}(\xi)\right),
\eeq
see \eqref{der6}. If we differentiate  \eqref{der2} with respect to $y$, then  for $p\geq 2$ we get
\[
\frac{\partial^{p}}{\partial y^p}B(x,y)=\mp q_\pm^{(p-1)}(x+y) +D_{\pm,p} (x+y) +\]
\[+\int^x_{x+y}q_\pm(\xi)\frac{\partial^{p-1}}{\partial y^{p-1}}
B_\pm(\xi,x+y-\xi)d\xi +\int_x^{\pm\infty} q_\pm(\xi)\frac{\partial^{p-1}}{\partial y^{p-1}}
B_\pm(\xi,y)d\xi,\]
where $D_{\pm,p}(\xi)$ is defined by \eqref{der9}. We complete the proof by induction taking into account \eqref{B-K} and the estimates \eqref{K1}, \eqref{K2} in which the exponent factors are replaced by the more crude estimate of type $C_{\pm}(x)$.
\end{proof}

\subsection{Analytical properties of the scattering data}

The spectrum of the Schr\"o- dinger operator $L$  with steplike potential \eqref{decay} consists of an absolutely continuous and a discrete part.
Using \eqref{nott} introduce the sets
\[
\Sigma^{(2)}:=[\overline{c}, +\infty), \quad
\Sigma^{(1)}:=[\underline{c}, \ \overline{c}], \quad
\Sigma =\Sigma^{(2)}\cup\Sigma^{(1)}.
\]
The set $\Sigma$ is the (absolutely) continuous spectrum of operator $L$, and
$\Sigma^{(1)}$ and $\Sigma^{(2)}$, are the parts
which are of multiplicities one and two, respectively. As mentioned in the introduction, we distinguish the points on the upper and lower sides of the set $\Sigma$.  Note that the set $\Sigma$ is the preimage of the real axis $\R$ under the conformal map $k_\pm(\la):\clos(\mathcal D) \to \ol{\C^+}$ when $c_\pm<c_\mp$. For $q \in\mathcal L_m^n(c_+,c_-)$ with $m\geq 1$ and $n\geq 0$, the operator $L$ has a finite discrete spectrum (see \cite{Akt}), which we denote as $\Sigma_d=\{\la_1,\dots,\la_p\}$,
where $\lambda_1<\cdots<\lambda_p<\underline{c}$. 
Our next step is to briefly describe some well-known analytical properties of the scattering data (\cite{BF}, \cite{CK}).
Most of these properties follow from analytical properties of the Wronskian  of the Jost solutions
$W(\la):=W\left(\phi_-(\la,\cdot),\,\phi_+(\la,\cdot)\right)$. The representations \eqref{Jost1} imply that the Jost solutions, together with their derivatives, decay exponentially fast as $x\to\pm\infty$ for $\im(k_\pm)>0$. Evidently, the discrete spectrum $\Sigma_d$ of $L$  coincides with the set of points, where $\phi_+$ is proportional to $\phi_-$ and,  their Wronskian vanishes.
The Jost solutions at these points are called the left and the right eigenfunctions. They are real-valued, and we denote the corresponding norming
constants by
\[
\gamma_{j}^{\pm}:=\left(\int_{\R}\phi_\pm^2(\la_j,x)dx\right)^{-1}.
\]

\begin{lemma}\label{propwronsk}
Let $q \in\mathcal L_m^n(c_+,c_-)$ with $m\geq 1$, $n\geq 0$. Then the function $W(\la)$ posseses the following properties
\begin{enumerate} [(i)]
 \item It is holomorphic in the domain $\mathcal D$ and continuous up to the boundary $\Sigma$ of this domain. Moreover, $W(\la+\I 0)=\overline{W(\la-\I 0)}\neq 0$ as $\la\in(\ul c, +\infty)$.
 \item It has simple  zeros in the domain $\mathcal D$ only at the points $\la_1,\dots,\la_p$, where 
\beq\label{zavis}
\left(\frac{ d W}{d\la}(\la_j)\right)^{-2}=\gamma_j^+\gamma_j^-.
\eeq

\end{enumerate}
\end{lemma}
Items (i)--(ii) are proved in \cite{bet} for $q\in\mathcal L_2^0(c_+, c_-)$, but the proof remains valid for $q\in\mathcal L_1^0(c_+, c_-)$. 
As we see, the only real value apart from the discrete spectrum where
the Wronskian can vanish, is the point $\ul c$. If $W(\ul c)=0$ we will refer to this as to the resonant case.

To study further the spectral properties of $L$, we consider the usual scattering relations
\begin{equation}\label{2.16}
T_\mp(\lambda) \phi_\pm(\lambda,x)
=\overline{\phi_\mp(\lambda,x)} + R_\mp(\lambda)\phi_\mp(\lambda,x),
\quad \mbox{as}\quad k_\pm(\la)\in\R,
\end{equation}
where the transmission and reflection coefficients are defined as usual,
\begin{equation}\label{2.17}
T_\pm(\lambda):= \frac{W(\overline{\phi_\pm(\lambda)}, \phi_\pm(\lambda))}{W(\phi_\mp(\lambda), \phi_\pm(\lambda))},\qquad
R_\pm(\lambda):= - \frac{W(\phi_\mp(\lambda),\overline{\phi_\pm(\lambda)})}{W(\phi_\mp(\lambda),\,\phi_\pm(\lambda))}, \quad k_\pm\in \R.
\end{equation}

Their  properties are  given in the following 

\begin{lemma}\label{lem2.3}
Let $q\in\mathcal L_m^n(c_+, c_-)$ with $m\geq 1$, $n\geq 0$. Then
the entries of the scattering matrix possess the following properties:
\begin{enumerate}[\bf I.]
\item
\begin{enumerate}[\bf(a)]
\item
$T_\pm(\lambda+\I 0) =\overline{T_\pm(\lambda-\I 0)}$ and
$R_\pm(\lambda+\I 0) =\overline{R_\pm(\lambda-\I 0)}$ for $k_\pm(\la)\in\R$.
\item
$\dfrac{T_\pm(\lambda)}{\overline{T_\pm(\lambda)}}= R_\pm(\lambda)$
for $\lambda\in\Sigma^{(1)}$ when $c_\pm=\underline{c}$.
\item
$1 - |R_\pm(\lambda)|^2 =
\dfrac{k_\mp}{k_\pm}\,|T_\pm(\lambda)|^2$ for $\lambda\in\Sigma^{(2)}$.
\item
$\overline{R_\pm(\lambda)}T_\pm(\lambda) +
R_\mp(\lambda)\overline{T_\pm(\lambda)}=0$ for $\lambda\in \Sigma^{(2)}$.
\item
$T_\pm(\lambda) = 1 +O(\la^{-1/2})$ and
$R_\pm(\lambda) = O(\la^{-1/2})$ for $\lambda\to\infty$.
\end{enumerate}
\item
\begin{enumerate}[\bf (a)] \item
The functions $T_\pm(\lambda)$ can be analytically continued to the domain
$\mathcal D$ satisfying
\begin{equation}\label{2.18}
2\I k_+(\la) T_+^{-1}(\la)  = 2\I k_-(\la) T_-^{-1}(\la) =:W(\la),
\end{equation}
where $W(\la)$ possesses the properties {\it (i)--(ii)} from Lemma~\ref{propwronsk}.
\item If $W(\ul c)=0$ then $W(\la)=\I\gamma\sqrt{\la - \ul c}\,(1+o(1))$, where $\gamma\in\R\setminus\{0\}$.
\end{enumerate}
\item $R_\pm(\la)$ is continuous for $k_\pm(\la)\in\R$. 
\end{enumerate}
\end{lemma}

\begin{proof}
Properties {\bf I.\ (a)--(e)}, {\bf II.\ (a)} are proved in \cite{bet} for $m=2$, and the proof remains valid for $m=1$. Property {\bf III} is evidently valid for $k_\pm\neq 0$ by  \eqref{2.17}, the continuity of the Jost solutions, and the absence of resonances.  Since $W(\ol c)\neq 0$ by Lemma \ref{propwronsk}, it remains to establish that in the case $\ul c=c_\pm$ the function $R_\pm$ is continuous as $k_\pm\to 0$. Since $\ol{\phi_\pm(c_\pm,x)}=
\phi_\pm(c_\pm,x)$, the property
\beq\label{minus1}
R_\pm(c_\pm)=-1\quad \mbox{if } \ W(c_\pm)\neq 0,
\eeq
follows immediately from \eqref{2.17}. In the resonant case, the proof of {\bf II.\ (b)} will be deferred to Subsection~\ref{march}.            
\end{proof}

Since we have deferred the proof of {\bf II.\ (b)}, we will not use it until then. However, we will need the following weakened version of property {\bf II.\ (b)}.

\begin{lemma}\label{lemweak}
If $W(\underline{c})=0$ then, in a vicinity of point $\ul c$, the Wronskian admits the estimates
\beq \label{musk1}
 W^{-1}(\la)=\left\{\begin{array}{ll}O\left((\la - \ul c)^{-1/2}\right) & \mbox{for}\  \ \la\in \Sigma,\\
 O\left((\la - \underline{c})^{-1/2 -\delta}\right)& \mbox{for}\  \la\in \C\setminus\Sigma,\end{array}
 \right.
\eeq
where $\delta>0$ is an arbitrary small number.
\end{lemma}

\begin{proof}
We give the proof for the case $c_-=\ul c$, $c_+=\ol c$. The other case is analogous. In this case the point $k_-=0$ corresponds to the point $\la=\ul c$. To study the Wronskian, we use \eqref{2.18} for $T_-(\la)$.
 First we prove that $T_-$ is bounded
on the set $V_\varepsilon:\{\la(k_-):-\varepsilon<k_-<\varepsilon\}$ for some  $\varepsilon >0$.
Indeed, due to the continuity of $\phi_+(\la,x)$ with respect to both variables, we can choose a point $x_0$ such that $\phi_+(\ul c, x_0)\neq 0$, respectively $|\phi_+(\la,x_0)|> \frac{1}{2}|\phi_+(\ul c,x_0)|>0$ in $V_\varepsilon$ for sufficiently small $\varepsilon$. Then by \eqref{2.16},
\[|T_-(\la)|=\frac{|R_-(\la)\phi_-(\la,x_0) +\ol{\phi_-(\la,x_0)}|}{|\phi_+(\la,x_0)|}\leq C,\quad \la\in V_\varepsilon.\]
Thus, for real $\la$ near $\ul c$ we have $W^{-1}(\la)=O((\la -\ul c)^{-1/2})$.
For non real $\la$ we use the fact that the diagonal of the kernel of the resolvent $(L-\la I)^{-1}$
\[
G(\la,x,x)=\frac{\phi_+(\la,x)\phi_-(\la,x)}{W(\la)},\ \la\in\mathcal D\setminus\Sigma_d,
\]
is a Herglotz--Nevanlinna function (cf.\ \cite{tschroe}, Lemma~9.22). Hence by virtue of Stieltjes inversion formula (\cite{tschroe}, Theorem~3.22)
it can be represented as
\[
G(\la,x_0,x_0)=\int_{\ul c}^{\ul c + \varepsilon^2} \frac{\im G(\xi+\I 0,x_0,x_0)}{\xi -\la}d\xi + G_1(\la),
\]
where $G_1(\la)$ is a bounded in a vicinity of $\ul c$. But $G(\xi+\I 0,x_0,x_0)=O((\xi -\ul c)^{-1/2})$ and by \cite[Chap.~22]{Musch} we get \eqref{musk1}.
 \end{proof}

In what follows we set $\kappa_j^\pm:=\sqrt{c_\pm -\la_j}$ such that $\I\kappa_j^\pm$ is the image of the eigenvalue $\la_j$ under the map $k_\pm$. Then we have the following
\begin{remark}
For the function $T_\pm(\la)$, regarded as a function of variable $k_\pm$,
\beq\label{rest}
\res_{\I\kappa_j^\pm} T_\pm(\la)=\I(\mu_j)^{\pm 1}\gamma_j^\pm,\ \ \text{where}
\ \ \phi_+(\la_j,x)=\mu_j\phi_-(\la_j,x).
\eeq
\end{remark}

\subsection{The Gelfand--Levitan--Marchenko equations}

Our next aim is to derive the Gelfand--Levitan--Marchenko equations. 
In addition to  {\bf I.\ (e)}, we will need another property of the reflection coefficients.

\begin{lemma}\label{rem7}
Let $q\in\mathcal L_1^0(c_+, c_-)$. Then the reflection coefficient $R_\pm(\la)$ regarded as a function of
$k_\pm\in\R$ belongs to the space $L^1(\R)=L^1_{\{k_\pm\}}(\R)$.
\end{lemma}

\begin{proof} 
Throughout this proof we will denote by $f_{s,\pm}:=f_{s,\pm}(k_\pm)$, $s=1,2,\dots$, functions whose Fourier transforms are in
$L^1(\R)\cap L^2(\R)$ (with respect to $k_\pm$). Note that $f_{s,\pm}$ are continuous.
Moreover, a function $f_{s,\pm}$ is continuous with respect to 
$k_\mp$ for $k_\mp=k_\mp(\la)$ with $\la\in\Sigma^{(2)}$, and  $f_{s,\pm}\in L^2_{\{k_\mp\}}(\R\setminus(-a,a))$, where the set $\R\setminus (-a,a)$ is the image of the spectrum $\Sigma^{(2)}$ under the map $k_\mp(\la)$.

Denote by a prime the derivative with respect to $x$. Then \eqref{Jost1}--\eqref{K2} and \eqref{defsig} imply
\[
\overline{\phi_\pm(\la,0)}= 1 + f_{1,\pm},\quad
\overline{\phi_\pm^\prime(\la,0)}=\mp\I k_\pm\,\overline{\phi_\pm(\la,0)} +f_{2,\pm},
\]
\[
\phi_\pm(\la,0)=1 +f_{3,\pm},\quad
\phi^\prime_\pm(\la,0)=\pm\I k_\pm\,\phi_\pm(\la,0) +f_{4,\pm}.
\]
Since
\beq\label{kapm}
k_\pm- k_\mp=\frac{c_\mp - c_\pm}{2 k_\pm}(1+o(1)) \ \mbox{ as }\ \ |k_\pm|\to\infty,
\eeq
then $W(\phi_\mp(\la),\ol{\phi_\pm(\la)})=f_{5,\pm}$ for large $k_\pm$.
By the same reason
\[
W(\la)=2\I\sqrt\la(1+o(1))\ \ \mbox{as}\ \ \la\to\infty.
\]
Remember that the reflection coefficient is a bounded function with respect to $k_\pm\in\R$ by {\bf I.\ (b), (c)}.
Moreover, for $|k_\pm|\gg 1$ it admits the representation $R_\pm(\la)=f_{6,\pm} k_\pm^{-1}$. This finishes the proof.
\end{proof}

\begin{lemma}
Let $q\in\mathcal L_1^0(c_+, c_-)$. Then the kernels of the transformation operators $K_\pm(x,y)$
satisfy the integral equations
\begin{equation}\label{ME}
K_\pm(x,y) + F_\pm(x+y) \pm \int_x^{\pm\infty} K_\pm(x,s) F_\pm(s+y)d s =0,
\quad \pm y>\pm x,
\end{equation}
where 
\begin{align}\label{lemF}
F_\pm(x)
=& \frac{1}{2\pi}\int_\R R_\pm(\la) e^{\pm\I k_\pm x}dk_\pm + \sum_{j=1}^p \ga_{j}^\pm \E^{\mp\kappa_j^\pm x}\\ \nn
& + \begin{cases}
 \frac{1}{4\pi}\int_{\underline{c}}^{\overline{c}} |T_\mp(\la)|^2\,|k_\mp|^{-1}\E^{\pm \I k_\pm x}d\la, & c_\pm=\ol c,\\
 0, & c_\pm=\ul c.
 \end{cases}
\end{align}
\end{lemma}

\begin{proof}
To derive the GLM equations, we introduce two functions
\[ G_\pm(\la,x,y) =
\left(T_\pm(\la) \phi_\mp(\la,x)-\E^{\mp \I k_\pm x}\right) \E^{\pm\I k_{\pm}y},\quad \pm y>\pm x,
\]
where $x,y$ are considered as parameters.  As a function of $\la$, both functions are
meromorphic in the domain $\mathcal D$ with simple poles
at the points $\lambda_j$ of the discrete spectrum. By property {\bf II}, they are continuous up
to the boundary $\Sigma^{\mathrm{u}}\cup\Sigma^{\mathrm{l}}$, except at the point $\ul c$, where one of these functions ($G_\mp(\la,x,y)$ for $\ul c=c_\pm$) can have a singularity of order $O((\la - \ul c)^{-1/2-\delta})$ in the resonant case by Lemma \ref{lemweak}.

By the scattering relations,
\begin{align*}
T_\pm(\la) \phi_\mp(\la,x)-\E^{\mp \I k_\pm x}&=R_\pm(\la)\phi_\pm(\la,x) + (\ol {\phi_\pm(\la,x)} - \E^{\mp \I k_\pm x})\\
&=S_{\pm,1}(\la,x) +
S_{\pm,2}(\la,x).
\end{align*}
It follows from \eqref{Jost1} that
\[
\frac{1}{2\pi}\int_{\R} S_{\pm,2}(\la,x)\E^{\pm\I k_{\pm}y}dk_\pm=K_\pm(x,y).
\]
Next, according to Lemma \ref{rem7} and \eqref{refff}, we obtain
\[
R_\pm(\la) K_\pm(x,s)\E^{\I k_\pm(y+s)}\in L^1_{\{k_\pm\}}(\R)\times L^1_{\{s\}}([x,\pm\infty))\quad\mbox{for } x,y\mbox{ fixed}.
\]
Using again \eqref{Jost1} and Fubini's theorem, we get
\begin{align*}
& \frac{1}{2\pi}\int_{\R} S_{\pm,1}(\la)\E^{\pm\I k_\pm y}dk_\pm  =\\
&\qquad = F_{r,\pm}(x+y) \pm \frac{1}{2\pi}\int_{\R}\int_x^{\pm\infty} K_\pm(x,s)R_\pm(\la)\E^{\pm\I k_\pm(y+s)}ds\, dk_\pm\\
&\qquad = F_{r,\pm}(x+y) \pm\int_x^{\pm\infty}K_\pm(x,s)F_{r,\pm}(y+s)ds,
\end{align*}
where we have set ($r$ for "reflection")
\beq\label{hatr}F_{r,\pm}(x):=\frac{1}{2\pi}\int_{\R} R_\pm(\la)\E^{\pm\I k_\pm x}dk_\pm.\eeq
Thus, for $\pm y>\pm x$,
\beq\label{int00}
\frac{1}{2\pi} \int_{\R}G_\pm(\la,x,y)dk_\pm=K_\pm(x,y) + F_{r,\pm}(x+y)\pm\int_x^{\pm\infty}K_\pm(x,s)F_{r,\pm}(y+s)ds.
\eeq
Now let $\mathcal C_\rho$ be a closed semicircle of radius $\rho$ lying in the upper half plane with the center at the origin
and set $\Gamma_\rho=\mathcal C_\rho \cup [-\rho, \rho]$. 
Estimates \eqref{decay5}, \eqref{zeros}, \eqref{kapm}, and {\bf I.\ (e)} imply that the Jordan lemma is applicable to the function 
$G_\pm(\la,x,y)$ as a function of $k_\pm$ when $\pm y\geq\pm x$. Moreover, formula \eqref{rest} implies 
 \[
 \phi_\mp(\la_j,x)\res_{\I\kappa_j^\pm}T_{\pm}(\la)=\I\gamma_j^\pm\phi_\pm(\la_j,x),
 \]
and thus
\begin{align}\nn
\sum_{j=1}^p \res_{\I\kappa_j^\pm}G_\pm(\la,x,y)&= \I\sum_{j=1}^p \gamma_j^\pm\phi_\pm(\la_j,x)\E^{\mp\kappa_j^\pm\,y}\\ \label{pomog}
& =\I\left(F_{d,\pm}(x+y) \pm\int_x^{\pm\infty}K_\pm(x,s) F_{d,\pm}(s+y)ds\right),
\end{align}
where we denote ($d$ for discrete spectrum)
\[
F_{d,\pm}(x):=\sum_{j=1}^p \gamma_j^\pm\E^{\mp\kappa_j^\pm\,x}.
\]
Now let  $c_\pm=\ul c$, which means that variable $k_\pm \in \R$ covers the whole continuous spectrum of $L$. Then the function  $G_\pm(\la,x,y)$ as a function of $k_\pm$ has a meromorphic continuation to the domain $\C^+$ with poles at the points  $\I\kappa_j^\pm$. By use of the Cauchy theorem, of the Jordan lemma and \eqref{int00},for $\pm x< \pm y$, we get 
\begin{align*}
\lim_{\rho\to\infty}\frac{1}{2\pi}\oint_{\Gamma_{\rho}} G_\pm(\la, x,y)dk_-& =\I\sum_{j=1}^p
\res_{\I\kappa_j^\pm} G_\pm(\la,x,y)=K_\pm(x,y)\\
&  + F_{r,\pm}(x+y)\pm\int_x^{\pm\infty}K_\pm(x,s) F_{r,\pm}(y+s)ds .
\end{align*}
Joining this with \eqref{pomog}, we get equation \eqref{lemF} in the case $c_\pm=\ul c$.  Unlike to this, in the case $c_\pm=\ol c$ the real values of variable $k_\pm$ correspond to the spectrum of multiplicity two only. Then the function $G_\pm(\la,x,y)$ considered as a function of $k_\pm$ in $\C^+$ has a jump along the interval  $[0, \I b_\pm]$ with $b_\pm =\sqrt{c_\pm - c_\mp}>0$.  It does not have a pole in $b_\pm$ because by Lemma \ref{lemweak}, the estimate  $G_\pm(\la,x,y)=O((k_\pm - b_\pm)^{\alpha})$ with $-1<\alpha\leq -1/2$ is valid.

For large $\rho>0$, put $b_\rho=b_\pm+\rho^{-1}$, introduce a union of three intervals
\[
\mathcal C_\rho^\prime=[-\rho^{-1},\I b_\rho -\rho^{-1}]\cup [\rho^{-1},\I b_\rho  +\rho^{-1}]\cup[\I  b_\rho  -\rho^{-1}, \I  b_\rho  +\rho^{-1}],
\]
and consider a closed contour $\Gamma_\rho^\prime=\mathcal C_\rho\cup\mathcal C_\rho^\prime\cup[-\rho, -\rho^{-1}]\cup[\rho^{-1},\rho]$ oriented counterclockwise. The function $G_\pm(\la,x,y)$ is meromorphic inside the domain bounded by $\Gamma_\rho^\prime$ (we suppose that $\rho$ is sufficiently large such that all poles are inside this domain). Thus,

 \begin{align}\label{secc}
 \lim_{\rho\to\infty}\frac{1}{2\pi}\oint_{\Gamma_{\rho}^\prime} G_\pm(\la, x,y)dk_\pm = & \I\sum_{j=1}^p
\res_{\I\kappa_j^\pm} G_\pm(\la,x,y)=K_\pm(x,y)\\ \nn
& {} + F_{r,\pm}(x+y)\pm\int_x^{\pm\infty}K_\pm(x,s) F_{r,\pm}(y+s)ds\\ \nn
& {} + \frac{1}{2\pi}\int_{\I b_\pm}^0 \left(G_\pm(\la+\I0,x,y) - G_\pm(\la-\I 0,x,y)\right)dk_\pm .
\end{align}
 In the case  under consideration, that is when $c_\pm=\ol c$, the variable $k_\pm=\I \kappa$, $\kappa>0$, does not have a jump along 
 the spectrum of multiplicity one, and the same is true for the solution $\phi_\pm(\la,x)$. Thus the jump $[G_\pm]:=G_\pm(\la+\I0,x,y) - G_\pm(\la-\I 0,x,y)$ 
  stems from the function $T_\pm(\la)\phi_\mp(\la,x)$. By \eqref{2.18} and {\bf I.\ (b)} we have 
 $T_\pm\overline {T_\pm^{-1}}= -T_\mp\overline {T_\mp^{-1}}=-R_\mp$ on $\Sigma^{(1)}$. To simplify notations we omit the dependence on 
 $\la$ and $x$. The scattering relations \eqref{2.17} then imply 
 \[
 T_\pm\phi_\mp - \ol{T_\pm\phi_\mp}=-\ol{T_\pm}\left(\ol{\phi_\mp} + R_\mp\phi_\mp\right)=
 -\ol{T_\pm}T_\mp\phi_\pm,
 \]
 and, therefore, $[G_\pm]=-\E^{\pm k_\pm y}\ol{T_\pm(\la+\I 0)}T_\mp(\la+\I0)\phi_\pm(\la,x)$.
Set 
\[
\chi(\la):=- \ol{T_\pm(\la+\I 0)}T_\mp(\la+\I0),\quad \la\in[\ul c, \ol c].
\]
By use of \eqref{Jost1}, we get
\[\aligned  &\frac{1}{2\pi}\int_{\I b_\pm}^0 \left(G_\pm(\la+\I0,x,y) - G_\pm(\la-\I 0,x,y)\right)dk_\pm\\ &=
F_{\chi,\pm}(x+y) \pm\int_x^{\pm\infty} K_\pm(x,s)F_{\chi,\pm}(s+y)ds,\endaligned\]
where
\[
F_{\chi,\pm}(x)=\frac{1}{2\pi}\int_{\I b_{\pm}}^0 \chi(\la)\E^{\pm\I k_\pm x}d k_\pm=
\frac{1}{4\pi}\int_{\ul c}^{\ol c} \chi(\la)\E^{\pm\I k_\pm x}\frac{d\la}{\sqrt{\la-c_\pm}}.\]
Combining this with \eqref{secc}, \eqref{pomog}, and \eqref{hatr} and taking into account that
by \eqref{2.18}, 
\[
\frac{\chi(\la)}{\sqrt{\la-c_\pm}}=|T_\mp(\la)|^2 |k_\mp|^{-1}>0,\ \ \la\in(\ul c, \ol c),
\]
gives \eqref{lemF} in the case $c_\pm=\ol c$.
\end{proof}

\begin{corollary}  Put $\hat F_\pm(x):=2F_\pm(2x)$. Then equation \eqref{ME}
reads
\beq \label{MEB} \hat F_\pm(x+y)+B_\pm(x,y)\pm\int_0^{\pm\infty} B_\pm(x,s)\hat F_\pm(x+y+s)ds=0, \eeq
where $B_\pm(x,y)$ is the transformation operator from \eqref{Jost11}.
\end{corollary}
This equation and Lemma \ref{Best} allows us to establish the decay properties of $F_\pm(x)$.
\begin{lemma}\label{lem4.1}
Let $q\in\mathcal L_m^n(c_+, c_-)$, $m\geq 1$, $n\geq 0$. Then the kernels of the GLM equations \eqref{ME} possess the  property:
\begin{enumerate}[\bf I.]
\addtocounter{enumi}{3}
\item
The function $F_\pm(x)$ is $n+1$ times differentiable with $F_\pm^\prime\in \mathcal L_{m}^n(\R_\pm)$.
\end{enumerate}
\end{lemma}

\begin{proof}
Differentiation of \eqref{MEB} $j$ times with respect to $y$ gives
\beq \label{derivGLM}
\hat F_\pm^{(j)}(x+y)+B^{(j)}_{\pm,y}(x,y)\pm\int_0^{\pm\infty} B_\pm(x,s)\hat F_\pm^{(j)}(x+y+s)ds=0.
\eeq
Set here $y=0$ and abbreviate $H_{\pm,j}(x)=B^{(j)}_{\pm,y}(x,0)$. Recall that the estimates \eqref{Best1} and \eqref{defsig1} imply  $H_{\pm,j}\in \mathcal L_m^{n+1-j}(\R_\pm)$, $j=1,\dots,n+1$. By changing variables $x+s=\xi$, we get
\beq \label{inner}
\hat F_\pm^{(j)}(x)+ H_{\pm,j}(x)\pm\int_x^{\pm\infty} B_\pm(x,\xi-x) \hat F_\pm^{(j)}(\xi)d\xi=0.
\eeq
Formula \eqref{B-K} and the estimate \eqref{K1} imply
\[
\abs{B_\pm(x,\xi-x)}\leq \sigma_{\pm,0}(\xi)
\E^{\hat\sigma_{\pm,0}(x)-\hat\sigma_{\pm,0}(\xi)}
\]
and from \eqref{inner} it follows
\begin{align} \label{in5} |\hat F_\pm^{(j)}(x)|\E^{-\hat\sigma_{\pm,0}(x)}\leq & |H_{\pm,j}(x)|\E^{-\hat\sigma_{\pm,0}(x)}\\ \nn
\pm &\int_x^{\pm\infty} \sigma_{\pm,0}(s) \E^{-\hat\sigma_{\pm,0}(s)}|\hat F_\pm^{(j)}(s)|ds\\ \nn
= & |H_{\pm,j}(x)|\E^{-\hat\sigma_{\pm,0}(x)} +\Phi_{\pm,j}(x),
\end{align}
where $\Phi_{\pm,j}(x):=\pm\int_x^{\pm\infty}|F_\pm^{(j)}(s)|\E^{-\hat\sigma_{\pm,0}(s)}\si_{\pm,0}(s)ds$.  Multiplying the last inequality by $\sigma_{\pm,0}(x)$ and using \eqref{defsig},
we get
\[
\mp\frac{d}{dx}(\Phi_{\pm,j}(x)\E^{-\hat\sigma_{\pm,0}(x)})\leq \abs{H_{\pm,j}(x)}\si_{\pm,0}(x)\E^{-2\hat\sigma_{\pm,0}(x)}.
\]
By integration, we have
\[
\Phi_{\pm,j}(x)\leq \pm C\E^{\hat\sigma_{\pm,j}(x)}\int_x^{\pm\infty} H_{\pm,j}(s)\si_{\pm,0}(s)ds.
\]
This inequality implies $\Phi_\pm(\cdot)\in \mathcal L_m^1(\R_\pm)$ because $H_{\pm,j}\in \mathcal L_m^{n+1-j}(\R_\pm)$, $j\geq 1$, $\sigma_{\pm,0}\in \mathcal L_{m-1}^1(\R_\pm)$.
Property {\bf IV} now follows from \eqref{in5}.
\end{proof}

\subsection{The Marchenko and Deift--Trubowitz conditions }\label{march}
In this subsection we give the proof of property {\bf II.\ (b)} and also prove the continuity of the reflection coefficient $R_\pm$ at the edge of the spectrum $\ul c$ when $c_\pm=\ul c$ in the resonant case. As is known, these properties are crucial for solving the inverse problem but were originally missed in the
seminal work of Faddeev \cite{Fa58} as pointed out by Deift and Trubowitz \cite{DT79}, who also gave a counterexample which showed that some
restrictions on the scattering coefficients at the bottom of the continuous spectrum were necessary for solvability of the inverse problem. The behavior of
the scattering coefficients at the bottom of the continuous spectrum is easy to understand for $m=2$, both for decaying and steplike cases, because the Jost solutions are differentiable with respect to the local parameters $k_\pm$ in this case. For $m=1$ the situation is more complicated. For the case $q\in\mathcal L_1^0(0,0)$ continuity of the scattering coefficients was established  independently by Guseinov \cite{gus} and Klaus \cite{Kl}
(see also \cite{KlAk}). For the case $q\in \mathcal L_1^0(c_+, c_-)$ property {\bf II.\ (b)} is proved in \cite{Akt}. We  propose here another proof following the  approach of Guseinov which will give as some additional formulas  of independent interest (in particular, when trying to understand the dispersive decay of
solutions to the  time-dependent Schr\"odinger equation, see e.g.\ \cite{EKMT}). Nevertheless, one has to emphasize that the Marchenko approach does not require these properties of the scattering data. In \cite{M} the direct/inverse scattering problem for $q\in \mathcal L_1^0(c_+, c_-)$ was solved under the following less restrictive conditions:
 \begin{enumerate} [1)]
\item The transmission coefficient $T(k)$, where $k^2=\la$, is bounded for $k\in \C^+$ in a vicinity of $k=0$ (at the bottom of the continuous spectrum);
 \item $\lim_{k\to 0} k T^{-1}(k)(R_\pm(k)+1)=0$.
 \end{enumerate}
Our properties {\bf I.\ (b)} and {\bf II} imply the Marchenko condition at point $\ul c$. Namely, if $W(\ul c)\neq 0$, then property 
(i) of Lemma \ref{propwronsk} implies $W(\ul c )\in\R$, and from {\bf I.\ (b)} it follows that $R_\pm(c_\pm)=-1$ for $\ul c=c_\pm$. 
The other reflection coefficient $R_\mp(c_-)$ is simply not defined at this point. Of course, it has the property $R_\mp(\ol c)=-1$ (cf.\ \eqref{minus1}), because $W(\ol c)\neq 0$, but we do not use this fact when solving the inverse problem. 
Our choice to give conditions {\bf I--III} as a part of necessary and sufficient ones is stipulated by the following. First of all, getting an analog of the Marchenko condition 1) directly, without {\bf II.\ (b)}, requires  additional efforts. The second reason is that in fact we 
 additionally justify here that the conditions proposed for $m=2$ in \cite{CK} are valid for the first finite moment of perturbation too.
 
The proof is given for the case $\ul c=c_-$, the case $\ul c=c_+$ is analogous.
For $k_\pm\in\R$ denote  $h_\pm(\la,x)=\phi_\pm(\la,x)\E^{\mp \I k_\pm x}$, then \eqref{Jost11} implies
\begin{align*}
h_\pm(\la) &= h_\pm(\la,0) = 1\pm\int_0^{\pm\infty} B_\pm(0,y)\E^{\pm2\I y k_\pm } dy,\\
h_\pm^\prime(\la) &= h_\pm^\prime(\la,0)=\pm\int_0^{\pm\infty}\frac{\pa}{\pa x} B_\pm(0,y)\E^{\pm2 \I  y k_\pm}dy.
\end{align*}
We observe that for $\ul c= c_-$ we have $2\I k_+(\ul c)= - b=- 2\sqrt{c_+ - c_-}<0$, and therefore,
in a vicinity of $\ul c$,
\beq\label{rright}
h_+(\la)= 1 +\int_0^\infty B_+(0,y)\E^{-b y}\E^{\I\tau(\la) y}dy, \quad \tau(\la)=2\frac{\la - \ul c}{k_+ -\I b/2},
\eeq
where $\tau(\la)$ is differentiable in a vicinity of $\ul c$ and $\tau(\ul c)=0$.
Since $ B_+(0,y)\E^{-b y}\in L_s^1(\R_+)$ and $ B_{+,x}(0,y)\E^{-b y}\in L_s^0(\R_+)$, $s=1,2,\dots$, then
\beq\label{imp}
\begin{aligned} -\phi_+(\ul c, 0) \phi_+^\prime(\la,0) + &\phi_+(\la, 0) \phi_+^\prime(\ul c,0)= h_+(\la)h_+^\prime(\ul c) - h_+(\ul c)h_+^\prime(\la)\\  
+& (2\I k_+ + b)h_+(\ul c)h_+(\la)= C(\la -\ul c)(1 + o(1)), \  \quad \la \to \ul c.\end{aligned}
\eeq
Now consider the function $\Phi(\la)=h_-(\la)h_-^\prime(\ul c) - h_-(\ul c)h_-^\prime(\la)$,
where $k_-\in\R$. One can show (cf.\ \cite{EKMT}) that it has a representation 
\beq\label{defFF}\Phi(\la)=2\I k_-\Psi(k_-), \ \ \mbox{where}\ \ \Psi(k_-)=\int_{\R_-}H(y)\E^{-2\I y k_-}dy,\eeq
with $H(x):=D(x)h_-(\ul c) - K(x) h^\prime_-(\ul c)$,
\[
K(x)=\int_{-\infty}^x B_-(0,y)dy,\ \ D(x)=\int_{-\infty}^x \frac{\pa}{\pa x} B_-(0,y)dy.
\]
Note that the integral in \eqref{defFF} is to be understood as an improper integral.
Using \eqref{MEB} and \eqref{derivGLM} one can get (see \cite{gus}) that the function $H(x)$ satisfies the following integral equation
\[
H(x) - \int_{\R_-} H(y)\hat F_-(x+y)dy=h_-(\ul c)\left(\int_{\R_-}B_-(0,y)\hat F_-(x+y)dy - F_-(x)\right).
\]
By property {\bf IV} we have $\hat F^\prime_- \in \mathcal L_1^0(\R_-)$. Using this and \eqref{K1} one can prove that $H\in L_1(\R_-)$ and therefore
$\Phi(\la)=2\I k_-\Psi(0)(1+o(1))$, with $\Psi(0)\in\R$. Moreover,
\[ \aligned
\phi_-(\la, 0) \phi_-^\prime(\ul c,0)- &\phi_-(\ul c, 0) \phi_-^\prime(\la,0)
=  -2\I k_-h_-(\la)h_-(\ul c) + \Phi(\la)\\ 
= & 2\I k_- (h_-(\ul c)^2 +\Psi(0))(1+O(1)),\quad \la\to \ul c,
\endaligned\]
where $h_-(\ul c)\in\R$. Combining this with \eqref{imp}, we get
the following

\begin{lemma}[\cite{Akt}]\label{lemg}
Let $\ul c=c_-$. Then in a vicinity of $\ul c$ the following asymptotics are valid:
\begin{enumerate}[(a)]
\item
If $\phi_-(\ul c, 0)\phi_+(\ul c,0)\neq 0$ then
\[
\frac{\phi_+^\prime(\la,0)}{\phi_+(\la,0)} -\frac{\phi^\prime_+(\ul c,0)}{\phi_+(\ul c,0)}=O(\la - \ul c),\quad \frac{\phi^\prime_-(\la,0)}{\phi_-(\la,0)} -\frac{\phi^\prime_-(\ul c,0)}{\phi_-(\ul c,0)}=\I\alpha\,\sqrt{\la - \ul c}(1+o(1));
\]
\item 
If $\phi_-^\prime(\ul c, 0)\phi_+^\prime(\ul c,0)\neq 0$ then
\[
\frac{\phi_+(\la,0)}{\phi_+^\prime(\la,0)} -\frac{\phi_+(\ul c,0)}{\phi_+^\prime(\ul c,0)}=O(\la - \ul c),\quad \frac{\phi_-(\la,0)}{\phi_-^\prime(\la,0)} -\frac{\phi_-(\ul c,0)}{\phi_-^\prime(\ul c,0)}=\I\hat \alpha\,\sqrt{\la - \ul c}(1+o(1)),
\]
\end{enumerate}
where $\alpha, \hat\alpha\in\R$.
\end{lemma}

Now suppose that $W(\ul c)= 0$, that is, $\phi_-(\ul c ,x)=C \phi_+(\ul c, x)$ with $C\in \R\setminus\{0\} $ being a constant. Therefore 
at least one of two cases described in Lemma \ref{lemg} holds true. Since the functions $\phi_+$ and $\phi_-$ are continuous in a vicinity of $\ul c$, then in the case (a) we have $\phi_-(\la,0)\phi_+(\la,0) = \beta (1+o(1))$ with $\beta\in\R\setminus\{0\}$. Thus,
\begin{align*}
W(\la) =& \phi_-(\la,0)\phi_+(\la,0)\left(\frac{\phi^\prime_-(\la,0)}{\phi_-(\la,0)} -\frac{\phi^\prime_-(\ul c,0)}{\phi_-(\ul c,0)}\right. \\
& -\left.\frac{\phi_+^\prime(\la,0)}{\phi_+(\la,0)} +\frac{\phi^\prime_+(\ul c,0)}{\phi_+(\ul c,0)}\right)=\I\alpha\beta\,\sqrt{\la - \ul c}(1+o(1)),
\end{align*}
where $\alpha\beta\in \R$. In fact, $\gamma=\alpha\beta\neq 0$ because of property \eqref{musk1}. The case (b) is analogous, and thus {\bf II.\ (b)} is proved. To prove
 the continuity of the reflection coefficient $R_-$ at $\ul c$ when $\ul c=c_-$ it is sufficient to apply a "conjugated" version of Lemma \ref{lemg}, which is valid if we consider the asymptotics as $\la\to\ul c$, $\la\in\Sigma^{(1)}$, to formula \eqref{2.17}.

We summarize our results by listing the conditions of the scattering data which  shown to be
necessary in the present section, and we will show them to be also sufficient for solving the inverse problem
in the next section.

\begin{theorem}[necessary conditions for the scattering data]\label{theor1}
The scattering data of a potential $q\in\mathcal{L}_m^n(c_+, c_-)$
\begin{align}\nn
{\mathcal S}_m^n(c_+, c_-): = \Big\{ & R_+(\lambda),\,T_+(\lambda),\, \sqrt{\la - c_+}\in\R; \,
R_-(\lambda),\,T_-(\lambda),\,   \sqrt{\la - c_-}\in \R;\\\label{4.6}
& \lambda_1,\dots,\lambda_p\in (-\infty, \ul c),\,
\gamma_1^\pm,\dots,\gamma_p^\pm\in\R_+\Big\}
\end{align}
possess the properties \emph{\textbf{I--III}} listed in Lemma~\ref{lem2.3}. The
functions $F_\pm(x,y)$, defined in \eqref{lemF}, possess property
\emph{\textbf{IV}} from Lemma~\ref{lem4.1}.
\end{theorem}

\section{The inverse scattering problem}

Let $\mathcal{S}_m^n(c_+, c_-)$ be a given set of data as in \eqref{4.6} satisfying the properties listed in Theorem~\ref{theor1}.

We begin by showing that, given $F_\pm(x,y)$ (constructed from our data via  \eqref{lemF}), the GLM equations \eqref{ME} can be
solved for $K_\pm(x,y)$ uniquely. First of all we observe that condition {\bf IV} implies 
$F_\pm\in \mathcal L_{m-1}^{n+1}(\R_\pm)$ (and therefore $F_\pm \in L^1(\R_\pm)\cap L^1_{\loc}(\R)$) as well as
$F_\pm$ is absolutely continuous on $\R$ for $m=1$.
Introduce the operator
\[
(\mathcal F_{\pm,x}f)(y)=\pm\int_0^{\pm\infty}  F_\pm(t+y+2x)f(t)dt.
\]
The operator is compact by \cite[Lem.~3.3.1]{M}. To prove that
$I+\mathcal F_{\pm,x}$ is invertible for every $x\in\R$, it is hence sufficient to prove
that the respective homogeneous equation
$
f(y)+\int_{\R_\pm}  F_\pm(y+t+2x) f(t)d t =0
$
has only the trivial solution in the space $L^1(\R_\pm)$. Consider first the case $\underline c=c_-$ and the equation
\beq \label{5.102} f(y) +\int_0^\infty  F_+(y+t+2x)f(t)dt=0,\quad f\in L^1(\R_+).\eeq Suppose that
 $f(y)$ is a
nontrivial solution of \eqref{5.102}.
Since $F_+(x)$ is real-valued, we can assume $f(y)$ to be real-valued too. By property {\bf IV}, the function $ F_+(t)$ is bounded as $t\geq x$ and hence the solution $f(y)$ is bounded too. Thus $f\in L^2(\R_\pm)$ and 
\begin{align} \nn
0 =& 2\pi\left(\int_{\R_+} f(y)\overline{f(y)} dy +\iint_{\R_+^2} F_+(y+t+2x)f(t)\overline{f(y)} dydt\right)=\sum_{j=1}^p \gamma_j^+(\tilde f(\la_j,x))^2\\ \nn
&+\int_{c_-}^{c_+}\frac{|T_-(\la)|^2}{|\la - c_-|^{1/2} }(\tilde f(\la,x))^2 d\la+\int_\R R_+(\la)\E^{2\I k x}\widehat f(-k)\widehat f(k)dk +\int_\R |\widehat f(k)|^2 dk,
\end{align}
where $k:=k_+=\sqrt{\la - c_+}$,
\[
\tilde f(\la,x)= \int_{\R_+} \E^{-\sqrt{c_+ - \la}\, (y+x)} f(y)d y, \quad  \mbox{and}\quad  \widehat f(k) =\int_x^\infty \E^{\I k y} f(y)d y .
\]
Since $\tilde f(\la,x)$ is real-valued for $\la<c_+$, the corresponding summands are nonnegative.
Omitting them and taking into account that (cf. \cite[Lem.~3.5.3]{M})
\[
\int_\R R_+(\la)\E^{2\I k x}\widehat f(-k)\widehat f(k) dk \leq \int_\R |R_+(\la)||\widehat f(k)|^2 dk,
\]
we come to the inequality $\int_\R (1-|R_+(\la)|)|\widehat f(k)|^2 dk\leq 0$. By property {\bf I.\ (c)}, $|R_+(\la)|<1$ for $\la\neq c_+$, therefore, $\widehat f(k)=0$, i.e.\ $f$ is the trivial solution of \eqref{5.102}.

For the solution $f$ of the homogeneous equation $(I+\mathcal F_{-,x})f=0$  we proceed in the same way and come to the inequality  $\int_\R (1-|R_-(\la)|)|\widehat f(k_-)|^2 dk_-\leq 0$, where $|R_-(\la)|<1$ for $\la >c_+$. Thus $\widehat f(k)$ is a holomorphic function for $k\in \C^+$,
continuous up to the boundary,  and $\widehat f(k)=0$  on the rays $k^2>c_+-c_-$. Continuing 
 $\widehat f(k)$ analytically in the symmetric domain $\C^+$ via these rays, we come to the equality  $\widehat f(k)=0$ for $k\in\R$.
 The case $\underline{c}=c_+$ can be studied similarly.
These considerations show that condition {\bf IV} can in fact be weakened:
 
\begin{theorem}\label{lemun}
Given $\mathcal S_m^n(c_+, c_-)$ satisfying conditions {\bf I--III}, let the function $F_\pm(x)$ be defined by \eqref{lemF}.
Suppose it satisfies the condition\\[2mm]
${\bf IV}^{\weak}$.  The function $F_\pm(x)$ is absolutely continuous  with $F_\pm^\prime\in L^1(\R_\pm)\cap L^1_{\loc}(\R)$. For any  $x_0\in \R$ there exists a positive continuous function $\tau_\pm(x, x_0)$, decreasing as $x\to\pm\infty$, with $\tau_\pm(\cdot, x_0)\in L^1(\R_\pm)$ and such that $|F_\pm(x)|\leq \tau_\pm(x, x_0)$ for $\pm x\geq \pm x_0$.\\[2mm]
Then 
\begin{enumerate}[(i)]
\item 
For each $x$, equation \eqref{ME} has a unique solution $K_\pm(x,\cdot)\in L^1([x,\pm\infty))$.
\item
This solution has first order partial derivatives satisfying 
\[
\frac{d}{dx} K_{\pm}(x,x)\in L^1(\R_\pm)\cap L^1_{\loc}(\R).
\]
\item The function
\beq\label{dopoln}
\phi_\pm(\la, x)= \E^{\pm\I k_\pm x}\pm\int_x^{\pm\infty} K_\pm(x,y)\E^{\pm\I k_\pm y}dy
\eeq
solves the equation
\[
-y^{\prime\prime}(x)\mp 2 y(x)\frac{d}{dx} K_{\pm}(x,x)= (k_\pm)^2 y(x),\quad x\in\R.
\]
\item 
If $F_\pm$ satisfies condition {\bf IV}, then $q_\pm(x):=\mp 2\frac{d}{dx} K_{\pm}(x,x)\in \mathcal L_m^n(\R_\pm)$.
\end{enumerate}
\end{theorem}

\begin{proof}
If $F_\pm$ satisfies condition {\bf IV} for any $m\geq 1$ and $n\geq 0$, then at least
 $F^\prime_\pm\in L_1^0(R_\pm)$, and we can choose $\tau_\pm(x,x_0)=\tau_\pm(x)=\int_{\R_\pm} |F^\prime(x+t)|dt$. Since $|F_\pm(x)|\leq \tau_\pm(x)$ and $\tau_\pm(\cdot)\in L^1(\R_\pm)$ is decreasing as $x\to\pm\infty$, condition ${\bf IV}^{\weak}$ is fulfilled.
 
Item (i) is already proved under the condition $F_\pm\in L^1(\R_\pm)\cap L^1_{\loc}(\R)$  and $F^\prime\in L^1_{\loc}(\R)$, which is weaker than ${\bf IV}^{\weak}$.  Therefore, we have a solution $K_\pm(x,y)$. To prove (ii) it is sufficient to prove $B^\prime_{\pm,x}=\frac{\pa}{\pa x}B_\pm(x,0)\in L^1[x_0,\pm\infty)$ for any $x_0$ fixed, where $B_\pm(x,y)=2K_\pm(x,x+2y)$.

 Let $\pm x\geq \pm x_0$.
Consider the GLM equation in the form \eqref{MEB}. By (i), the operator $I+\hat{\mathcal F}_{\pm, x}$ generated by the kernel $\hat F_\pm$ is also invertible and admits estimate $\|\{I+\hat {\mathcal F}_{x,\pm}\}^{-1}\|\leq C_\pm(x)$, where $C_\pm(x)$, $x\in\R$ is a continuous function with $C_\pm(x)\to 1$ as $x\to \pm \infty$.
Introduce the notations
\[\tau_{\pm,1}(x)=\int_{\R_\pm}|\hat F_\pm^\prime(t+x)|dt,\  \ \tau_{\pm,0}(x)=\int_{\R_\pm}|\hat F_\pm(t+x)|dt.
 \]
Note that $\abs{\hat F_\pm(x)}\leq \tau_{\pm,1}(x)$. From the other side, 
$\abs{\hat F_\pm(x)}\leq 2\tau_{\pm}(2x,2x_0)$, where $\tau_{\pm}(x,x_0)$ is the function from
condition ${\bf IV}^{\weak}$. From \eqref{MEB}, we have
\beq \label{ift} \int_{\R_\pm}\abs{B_\pm(x,y)}dy\leq
\norm{\{I+\hat{\mathcal F}_{\pm,x}\}^{-1}} \int_{\R_\pm} \abs{\hat F_\pm(y+x)}dy\leq C_\pm(x)\tau_{\pm,0}(x),\eeq
and, therefore,
\begin{align}\label{imp10}
|B_\pm(x,y)|\leq  &|\hat F(x+y)| +\int_{\R_\pm}|B_\pm(x,s)\hat F(x+y+s)|ds\\ \nn
\leq &\tau_\pm(2x+2y, 2x_0)(1+C_\pm(x)\tau_{\pm,0}(x))\leq C(x_0)\tau_\pm(2x+2y, 2x_0).
\end{align}
Being the solution of \eqref{MEB} with absolutely continuous kernel $\hat F_\pm$, the function
$B_\pm(x,y)$ is also absolutely continuous with respect to $x$ for every $y$.
Differentiate \eqref{MEB} with respect to $x$. Proceeding as in \eqref{ift}, we get then
\begin{align}\nn
\int_{\R_\pm}\abs{B_{\pm,x}^\prime(x,y)}dy\leq &
\norm{\{I+\hat{\mathcal F}_{\pm,x}\}^{-1}} \left( \int_{\R_\pm}\int_{\R_\pm}|B_\pm(x,t)\hat F^\prime(t+y+x)|dt dy\right.\\ +
& \left.\int_{\R_\pm} \abs{\hat F_\pm^\prime(y+x)}dy\right)\leq C_\pm(x)\left(\tau_{\pm,0}(x) + C_\pm(x)\tau_{\pm,1}(x)\tau_{\pm,0}(x)\right).\label{ifneed}
\end{align}
Now set $y=0$ in the derivative of \eqref{MEB} with respect to $x$. By use of \eqref{ift},
\eqref{ifneed} and ${\bf IV}^{\weak}$, we have then
\begin{align} \nn
|\hat F_\pm^\prime(x)+B_{\pm,x}^\prime(x,0)|\leq & \int_{\R_\pm} |B_{\pm,x}^\prime(x,t)\hat F_\pm(t+x)|dt
+\int_{\R_\pm} |B_\pm(x,t)\hat F_\pm^\prime(t+x)|dt\\
\nn \leq & C_\pm(x)(1+C_\pm(x)\tau_{\pm,1}(x))\tau_{\pm,0}(x) \tau_\pm(2x,2x_0) + H_\pm(x),
\end{align}
where $H_\pm(x)=\int_{\R_\pm}|B_\pm(x,t)\hat F_\pm^\prime(x+t)|dt$. By \eqref{imp10},
\[
H_\pm(x)\leq C(x_0) \int_{\R_\pm}\tau_\pm(2x +2t,2x_0)|\hat F_\pm^\prime(x+t)|dt\leq C(x_0)\tau_\pm(2x, 2x_0)\tau_{\pm,1}(x),
\]
which implies
\beq
|B^\prime_{\pm,x}(x,0)|\leq |\hat F^\prime(x)| +C(x_0)\tau_{\pm,1}(x) \tau_\pm(2x,2x_0).\label{main16}
\eeq
Therefore, under condition ${\bf IV}^{\weak}$, we get
$q_\pm(x):=B_{\pm,x}(x,0)\in L^1(\R^\pm)\cap L^1_{\loc}(\R)$, which proves (ii).

Repeating literally the corresponding part of the proof for Theorem 3.3.1 from \cite{M} we get item (iii) under condition ${\bf IV}^{\weak}$.

Now let $\hat F_\pm$ satisfy condition {\bf IV} for some $m\geq 1$ and $n\geq 0$. As
we already discussed, in this case one can replace $\tau_\pm(x,x_0)$ by $\tau_{\pm,0}(x)$, and then formulas \eqref{main16} and \eqref{imp10} read
\[
|B_\pm(x,y)|\leq C(x_0)\tau_{\pm,1}(x+y),\quad |B_{\pm,x}(x,0)|\leq C(x_0)\tau_{\pm,1}^2(x).
\]
Since $\tau_{\pm,1}(x)\in \mathcal L_{m-1}^1(\R_\pm)$ and $\tau_{\pm,1}^2(x)\in \mathcal L_m^0(\R_\pm)$ for $m\geq 1$, then $q_\pm(x)\in \mathcal L_m^0(\R_\pm)$. 
To prove the claim for higher derivatives, we proceed similarly. Namely, in agreement with previous notations, set
\[
\tau_{\pm,i}(x):=\int_{\R_\pm}\hat F_\pm^{(i)}(t+x)dt, \quad i=0,\dots,n+1,
\]
and also denote 
$D_{\pm}^{(i)}(x,y):=\frac{\pa^i}{\pa x^i}B_\pm(x,y)$. Denote by ${i\choose j}$ the binomial coefficients.
Differentiating \eqref{MEB} $i$ times with respect to $x$ implies
\[
\hat F_\pm^{(i)}(x+y)+D_{\pm}^{(i)}(x,y)= -\sum_{j=0}^{i}{ i\choose j} \int_{\R_\pm}
\hat F_\pm^{(j)}(x+y+t)D_\pm^{(i-j)}(x,t)dt,\] and, therefore, \begin{align}\nn
\int_{\R_\pm} \abs{D_\pm^{(i)}(x,y)}dy
 \leq&\| \{I+\hat F_{\pm,x}\}^{-1}\| \left \{\int_{\R_\pm} |\hat F_\pm^{(i)}(x+y)|dy \right .
 \\ \nn & \left . \sum_{j=1}^{i} { i\choose j} \int_{\R_\pm}\int_{\R_\pm} |\hat F_\pm^{(j)}(x+y+t)D_\pm^{(i-j)}(x,t)|dt dy\right\} \\ \nn
 \leq&  C_{\pm,i}(x)[\tau_{\pm,i-1}(x) +\sum_{j=1}^{i} \tau_{\pm,j}(x) \rho_{\pm,i-j}(x)],\end{align}
where $C_{\pm,i}(x):= K_i\|\{I+F_{\pm,x}\}^{-1}\| =K_i C_\pm(x)$ with $K_i=\max_{j\leq i} { i\choose j}$, and $\rho_{\pm,j}(x)$ is defined by the recurrence formula 
\[
\rho_{\pm,0}(x):= C_\pm(x)\tau_{\pm,0}(x),\ \ \rho_{\pm,s}:= C_{\pm,s}(x)[\tau_{\pm,s-1}(x)+\sum_{j=1}^{s} \tau_{\pm,j}(x) \rho_{\pm,s-j}(x)].
\] 
Thus, for every $i=1,\dots,n+1$,
\[
\int_{\R_\pm}|D^{(i)}_\pm(x,y)|dy\leq \rho_{\pm,i}(x)\in \mathcal L_{m-1}^0(\R_\pm).
\]
Respectively,
\[
|q_\pm^{(i)}(x)|= |D_\pm^{(i)}(x,0)|\leq
|F^{(i)}(x)| + \sum_{j=1}^{i} { i\choose j}\tau_{\pm,j}(x)\rho_{\pm,i-j}(x)\in \mathcal L_m^0(\R_\pm),
\]
which finishes the proof.
\end{proof}

Our next aim is to prove that the two functions $q_+(x)$ and $q_-(x)$ from the previous theorem do in fact coincide.

\begin{theorem}\label{theor2}
Let the set  $\mathcal S_m^n(c_+,c_-)$ defined by \eqref{4.6} 
satisfy conditions \emph{\textbf{I}--\textbf{III}} and ${\bf IV}^{\weak}$. Then $q_-(x)\equiv q_+(x)=:q(x)$.
If $\mathcal S_m^n(c_+,c_-)$ satisfies conditions \emph{\textbf{I}--\textbf{IV}} then $q\in\mathcal L_m^n(c_+, c_-)$.
\end{theorem}

\begin{proof}
This proof is a slightly modified version of the proof proposed in \cite{M}. We give it for the case $\ul c=c_-$. We continue to use the notation $\Sigma^{(2)}$ for the two sides of the cut along the interval $[\ol c,\infty)=[c_+,\infty)$, the notation $\Sigma$ for the two sides of the cut along the interval $[\ul c, \infty)=[c_-,\infty)$ and we also keep the  notation $\mathcal D=\C\setminus\Sigma$.

The main differences between the present proof and that from \cite{M} concern the presence of the spectrum of multiplicity one and the use of condition ${\bf IV}^{\weak}$. Namely, recall that the kernels of the GLM equations \eqref{ME} can be split naturally into the  summands
  $F_+=F_{\chi,+}+F_{d,+}+F_{r,+} $  and
$F_-=F_{r,-}+F_{d,-}$  according to \eqref{lemF}.

We begin by considering a part of the GLM equations
\[
G_\pm(x,y):=F_{r,\pm}(x+y)
\pm\int_x^{\pm\infty}K_\pm(x,t)F_{r,\pm}(t+y)d t,
\]
where
$K_\pm(x,y)$ are the solutions of GLM equations obtained in Theorem \ref{lemun}.  By condition
${\bf IV}^{\weak}$, we have $F_{r,\pm}\in L^2(\R)$, therefore, for any fixed $x$,
\[
\int_\R
F_{r,\pm}(x+y) \E^{\mp\I y k_\pm}d y =R_\pm(\la)\E^{\pm\I x k_\pm},
\]
and,  consequently, 
\begin{equation}\label{5.7}
\int_\R G_\pm(x+y) \E^{\mp\I k_\pm y}d y =R_\pm(\la)\phi_\pm(\la,x),
\quad k_\pm\in\R,
\end{equation}
where $\phi_\pm$ are the functions obtained in Theorem \ref{lemun}, and the integral is considered as a principal value.
On the other hand, invoking the GLM equations and the same functions $\phi_\pm$, we have 
\begin{align*}
G_+(x,y)
&=  -K_+(x,y) - \sum_{j=1}^p \gamma_j^+ \E^{-\kappa_j  y} \phi_+(\lambda_j,x) \\ 
& \quad -\frac{1}{4\pi}\int_{\ul c}^{\ol c} \frac{|T_-(\xi)|^2 }{k_-(\xi)}
\E^{\I k_+(\xi) y} \phi_+(\xi,x) d\xi,\quad y>x,\end{align*}
and 
\[
G_-(x,y)
= -K_-(x,y) + \sum_{j=1}^p \gamma_j^- \E^{\kappa_j y} \phi_-(\la_j,x),\quad y<x.
\]
Since for two points $k^\prime\neq k^{\prime\prime}$ \[\int_x^{\pm\infty} \E^{\pm\I\left( k^\prime - k^{\prime\prime}\right)y} dy=
\I\frac{\E^{\pm\I\left( k^\prime - k^{\prime\prime}\right)x}}{k^\prime - k^{\prime\prime}} ,\] then 
 \begin{align}\label{Gplus}
 &\int_\R G_+(x,y) \E^{-\I k_+ y} d y=
 \int^x_{-\infty}G_+(x,y) \E^{-\I k_+ y} d y
 -\int_x^{+\infty}K_+(x,y)\E^{-\I k_+ y} d y\\ \nn
  +&\frac{1}{4\pi\I}\int_{\ul c}^{\ol c} \frac{|T_-(\xi)|^2 \phi_+(\xi, x)\E^{\I (k_+(\xi) -k_+(\la)) x}}{
(k_+(\xi) - k_+(\la))\sqrt{\xi - c_-}} d\xi + \sum_{j=1}^p\,\gamma_j^+ \phi_+(\la_j,x)\,\frac{\E^{(-\I k_+ -\kappa_j^+) x}}
{\kappa_j^+  +\I k_+},
 \end{align}
 and
  \begin{align}\label{Gminus}
 \int_\R G_-(x,y) \E^{\I k_- y} d y= &
 \int_x^{+\infty}G_-(x,y) \E^{\I k_-y} d y
 -\int^x_{-\infty}K_-(x,y)\E^{\I k_- y} d y\\ \nn
 & +
 \sum_{j=1}^p\,\gamma_j^- \phi_-(\la_j, x)\,\frac{\E^{(\I k_- + \kappa_j^-) x}}
{\kappa_j^-  + \I k_-}\,.
 \end{align}
Since for $k_\pm\in\R$
\[
\pm\int_x^{\pm\infty} K_\pm(x,y)\E^{\mp \I k_\pm y}dy=\overline{\phi_\pm(\la,x)} - \E^{\mp \I k_\pm x},
\]
then, combining \eqref{Gplus} and \eqref{Gminus} with
 \eqref{5.7},  we infer the relations
 \begin{equation}\label{5.10}
 R_\pm(\lambda)\,\phi_\pm(\lambda,x) +
 \overline{\phi_\pm(\lambda,x)} =T_\pm(\lambda)\theta_\mp(\lambda,x),
 \quad k_\pm\in\R,
 \end{equation}
 where
 \begin{align}\nn
\theta_-(\lambda,x) &:=\frac{1}{T_+(\lambda)}\left(
\E^{-\I k_+  x} +\int^x_{-\infty}G_+(x,y)
\E^{-\I k_+ y} d y \right.\\ \label{5.11} &\qquad
-\int_{c_-}^{c_+}  \frac{|T_-(\xi)|^2 W_+(\xi,\la,x)}{
4\pi (\xi-\la)\sqrt{\xi - c_-}} d\xi+\left. \sum_{j=1}^p\,\gamma_j^+ \,\frac{W_+(\la_j,\la,x)}
{\la - \la_j}\right),\\ \nn
\theta_+(\lambda,x) &:=\frac{1}{T_-(\lambda)}\left(
\E^{\I k_- x} +\int_x^{+\infty}G_-(x,y)
\E^{\I k_- y}d y +\sum_{j=1}^p \gamma_j^- 
\frac{W_-(\la_j,\la,x)}{\lambda
- \lambda_j}\right),
\end{align}
and
\beq\label{dopp}
W_\pm(\xi,\la,x):=\I\phi_\pm(\xi,x)\E^{\pm\I (k_\pm(\xi) -k_\pm(\la)) x}( k_\pm(\xi) + k_\pm(\la)).
\eeq
It turns out that in spite of the fact that $\theta_\pm(\lambda,x)$ is defined via the
background solutions corresponding to the opposite half-axis $\R_\mp$, it
shares a series of properties with $\phi_\pm(\lambda,x)$.

\begin{lemma} \label{3.3}
The function $\theta_\pm(\la,x)$ possesses the following properties:
\begin{enumerate}[(i)]
\item
It admits an analytic continuation to the set $\mathcal D \setminus \{ c_+, c_- \}$ and is continuous up to its boundary $\Sigma$. 
\item It has no jump along the interval
$(-\infty, c_\pm]$, and it takes complex conjugated values
on the two sides of the cut along $[c_\pm,\infty)$.
\item For large $\la\in\clos(\mathcal D)$ it has the asymptotic behavior
$\theta_\pm(\la,x)= \E^{\pm\I k_\pm x}(1 + o(1))$.
\item The formula $W\left(\theta_\pm(\la,x),\phi_\mp(\la,x)\right)=\mp W(\la)$ is valid for $\la\in\clos(\mathcal D)$,
where $W(\la)$ is defined by formula \eqref{2.18}.
\end{enumerate}
\end{lemma}

\begin{proof}
The function  $T_\mp^{-1}(\la)$ 
admits an analytic continuation to $\mathcal D$ by property {\bf II.\ (a)}. Moreover, we have
$G_\mp(x,\cdot)\in L^1([x,\pm\infty))$. Since $\E^{\pm \I k_\mp y}$ does not grow as $\pm y\geq 0$  then the respective integral (the second summand in the representation for $\theta_\pm$) admits analytical continuation  also. The function $\theta_\pm$ does not have singularities at points $\{\la_1,\dots, \la_p\}$ since $T^{-1}_\mp(\la)$ has simple zeros at $\la_j$. The function $W_\mp(\xi,\la,x)$ can be continued analytically with respect to $\la$ for $\xi$ and $x$ fixed. Next, consider the Cauchy type integral term in \eqref{5.11}. The only singularity of the integrand can appear at point $\ul c= c_-$, because   in the resonance case $T_-(c_-)\neq 0$. 
 Thus, if $W(c_-)=0$, then the integrand in \eqref{5.11} behaves as $O(\xi -c_-)^{-1/2}$. By \cite{ Musch}, the integral is of order  $O(\xi -c_-)^{-1/2-\delta}$ for arbitrary small positive $\delta$, 
 moreover, $T_+^{-1}(\la)=C\sqrt{\la - c_-}(1+o(1))$. Therefore for $\la \to c_-$
\beq\label{thetaminus}
\theta_-(\la,x)=\begin{cases} O((\la -c_-)^{-\delta}), &  \mbox{if}\ W(c_-)=0,\\ O(1),& \mbox{if}\ W(c_-)\neq 0.\end{cases}
\eeq
Since $W(c_+)\neq 0$ by {\bf II.\ (a)}, then $T_+^{-1}(\la)=O(\la - c_+)^{-1/2}$, respectively
\beq\label{vazhn}
\theta_-(\la,x)= O\left( (\la - c_+)^{-1/2}\right),\quad \theta_+(\la,x)=O(1),\quad \la\to c_+.
\eeq
Properties (i) of Lemma \ref{propwronsk}, and {\bf II.\ (a)} together with \eqref{5.11} and \eqref{dopp} imply that $\theta_+$ and $\theta_-$ take complex conjugated values on the sides of cut along $[\ul c, \infty)$. Since $W_\pm(\xi,\la,x)\in\R$ when $\la, \xi\leq c_\pm$, then $\theta_\pm(\la,x)\in\R$ as $\la\leq c_-$.
Due to property {\bf I.\ (b)}, we have $ \overline{ T_-^{-1}} T_- = R_-$ on both sides of cut along $[\ul c, \ol c]$, and from \eqref{5.10} it follows that
\[
\theta_+=\phi_-\,\overline { T_-^{-1}}+\overline{\phi_-}\,T_-^{-1}\in\R.
\]
Therefore, $\theta_+$ has no jump along the interval $[\ul c, \ol c]$. At the point $\ul c=c_-$, the function $\theta_+(x,\la)$ has an isolated 
nonessential singularity, i.e. a pole at most. But at the vicinity of point $c_-$  $\theta_+(\la,x)=O(T_-^{-1}(\la))=O(\la - c_-)^{-1/2}$. Thus this singularity is removable,
\beq\label{vazhn2}
\theta_+(\la,x)=O(1),\quad \la\to c_-.
\eeq
Items (i) and (ii) are proved. 

The main term of asymptotical behavior for $\theta_\pm(\la,x)$ as $\la\to\infty$ is the first summand in \eqref{5.11}.
Thus, by {\bf I.\ (e)} and \eqref{kapm},
\[
\theta_\pm(\la,x)=T_\mp^{-1}(\la)\E^{\pm \I k_\mp\, x} +o(1)=\E^{\pm \I k_\pm\, x}(1 +o(1)), 
\]
which proves (iii). Property (iv) follows from \eqref{5.10}, \eqref{dopoln}, and \eqref{2.18} by analytic continuation.
\end{proof}

Now conjugate equality \eqref{5.10} and eliminate $\overline{\phi_\pm}$ from the system
\[
\left\{\begin{array}{ll}\ol{R_\pm\phi_\pm} + \phi_\pm & =\ol{\theta_\mp T_\pm},\\
R_\pm\phi_\pm +\ol\phi_\pm & =\theta_\mp T_\pm,
\end{array}\right. \quad k_\pm\in \R,
\]
to obtain
\[
\phi_\pm(1 - |R_\pm|^2)=\ol{\theta_\mp}\ol{T_\pm} - \ol{R_\pm}\theta_\mp T_\pm.
\]
Using    {\bf I.\ (c), (d)} and {\bf II} shows for $\la\in\Sigma^{(2)}$, that is for $k_+\in \R$, that
\[
T_\mp\phi_\pm=\ol{\theta_\mp} + R_\mp\theta_\mp\quad \la\in\Sigma^{(2)}.
\]
This equation together with \eqref{5.10} gives us a system from which we can eliminate the reflection coefficients $R_\pm$. We get
\beq\label{7new}
T_\pm(\phi_\pm\phi_\mp - \theta_\pm\theta_\mp)=\phi_\pm\ol\theta_\pm - \ol\phi_\pm\theta_\pm,\quad \la\in\Sigma^{(2)}.
\eeq
Next introduce a function
\[
\Phi(\la):=\Phi(\la,x)=\frac{\phi_+(\la,x)\phi_-(\la,x) - \theta_+(\la,x)\theta_-(\la,x)}{W(\la)},
\]
which is analytic in the domain $\clos(\mathcal D)\setminus\{\la_1,\dots,\la_p,\ul c, \ol c\}$.  Our aim is to prove that this function has no jump along the real axis and has removable singularities at the points $\{\la_1,\dots,\la_p, \ul c, \ol c\}$. Indeed, from \eqref{7new} and \eqref{2.18} we see that 
\[
\Phi(\la)=\pm\frac{\phi_\pm(\la,x)\ol{\theta_\pm(\la,x)} - \ol{\phi_\pm(\la,x)}\theta_\pm(\la,x)}{2\I k_\pm},\quad \la\in\Sigma^{(2)}.
\]
By the symmetry property (cf.\ {\bf II.\ (a)}, (iii), Theorem \ref{lemun} and (ii), Lemma \ref{3.3}), we observe that both the nominator and denominator
are odd functions of $k_+$, therefore $\Phi(\la+\I 0)=\Phi(\la -\I 0)$, as $\la\geq\ol c$, i.e., the function $\Phi(\la)$ has no jump along this interval.
By the same properties  {\bf II.\ (a)}, (iii) of Theorem \ref{lemun}  and (ii) of Lemma \ref{3.3} the function $\Phi(\la)$ has no jump on the interval $\la \leq \ul c$ as well.  Let us check that it has no jump along the interval $(\ul c, \ol c)$ also. Lemma \ref{3.3}, (ii) shows that the function $\theta_+(\la,x)$ has no jump here. Abbreviate
\[
[\Phi]=\Phi(\la+\I 0) - \Phi(\la-\I 0)=\phi_+\left[\frac{\phi_-}{W}\right]-\theta_+
\left[\frac{\theta_-}{W}\right],\quad \la\in(\ul c, \ol c),
\] 
and drop some dependencies for notational simplicity. Using property {\bf I , (b)} and formula \eqref{5.10}, we get
\[
\left[\frac{\phi_-}{W}\right]=\frac{\phi_-T_- + \ol{\phi_-T_-}}{2\I k_-}=\frac{(\phi_-R_- + \ol{\phi_-})\ol T_-}{2\I k_-}=\frac{\theta_+ T_-\ol T_-}{2\I k_-},
\]
that is, 
\beq\label{new10}
\phi_+\left[\frac{\phi_-}{W}\right]=\frac{\theta_+\phi_+ |T_-|^2 }{2\I k_-}.
\eeq
On the other hand, since $\I k_+\in\R$ as $\la<\ol c$, we have
\beq\left[\frac{\theta_-}{W}\right]=
\left[\frac{\theta_-T_+}{2\I k_+}\right]=\frac{1}
{2\I k_+}\left[\theta_-T_+\right].\label{9new}
\eeq
By \eqref{5.11} the jump of this function appears from the Cauchy type integral only. Represent this integral as
\[
-\frac{1}{2\pi\I}\int_{\ul c}^{\ol c}\frac{\phi_+(x,\xi)(-\I)(k_+(\la) + k_+(\xi))\E^{\I x(k_+(\xi) -k_+(\la))}|T_-(\xi)|^2}{2\I k_-(\xi)}\frac{d\xi}{\xi - \la},
\]
and apply the Sokhotski--Plemejl formula. Then \eqref{9new} implies
\[
\theta_+\left[\frac{\theta_-}{W}\right]=\frac{\theta_+\phi_+|T_-|^2}{2\I k_-}.
\]
Comparing this with \eqref{new10}, we conclude that the function $\Phi(\la)$ has no jumps on $\C$, but may have isolated singularities at the points $E=\la_1,\dots,\la_p, c_-,c_+ $ and $\infty$. Since all these singularities are at most isolated poles, it is sufficient to check that $\Phi(\la)=o((\la - E)^{-1})$, from some direction in the complex plane, to show that they are removable. First of all, properties {\bf I.\ (e)} and (iii), Lemma \ref{3.3} together with \eqref{2.18} and \eqref{dopoln} imply $\Phi(\la)\to 0$ as $\la\to\infty$.  The desired behavior $\Phi(\la)=o((\la - c_\pm)^{-1})$ for $\la\to c_\pm$ is due to property {\bf II} and estimates \eqref{thetaminus}, \eqref{vazhn}, \eqref{vazhn2}.
Next, to prove that there is no singularities at the points of the discrete spectrum, we have to check that
\beq\label{new11}
\phi_+(x,\la_j)\phi_-(x,\la_j)=\theta_+(x,\la_j)\theta_-(x,\la_j).
\eeq
Passing to the limit in both formulas \eqref{5.11} and taking into account \eqref{2.18} and \eqref{dopp} gives
\[
\theta_\mp(\la_k,x)=\frac{d W}{d\la} (\la_k)\,\phi_\pm(\la_k,x)\,\gamma^\pm_j,
\]
which together with \eqref{zavis} implies \eqref{new11}.
Since $\Phi(\la)$ is analytic in $\C$ and $\Phi(\la)\to 0$ as $\la \to\infty$, Liouville's theorem shows
\beq\label{equiv}
\Phi(x,\la)\equiv 0\quad\mbox{ for}\ \ \la\in\C,\ \ x\in\R.
\eeq

\begin{corollary}
$R_\pm(c_\pm)=-1$ if $W(c_\pm)\neq 0$.
\end{corollary}

\begin{proof}
In the case $\ul c=c_-$ discussed above we have $W(c_+)\neq 0$. Formula \eqref{equiv} implies that instead of \eqref{vazhn} we have in fact $\theta_-(x,\la)=O(1)$ as $\la\to c_+$. Since $T_+(c_+)=0$ and $\phi(x, c_+)=\ol {\phi(x, c_+)}$, then by \eqref{5.10} we conclude $R_+(c_+)=-1$.
Property $R_-(c_-)=-1$ in the nonresonant case is due to {\bf I.\ (b)}, \eqref{2.18}, and property $W(c_-)\in \R\setminus\{0\}$, which follows in turn from the symmetry property (i) of Lemma \ref{propwronsk}.
\end{proof}

Formula \eqref{equiv} implies  
 \beq\label{vazhn10}
 \phi_+(\la,x)\phi_-(\la,x)=\theta_+(\la,x)\theta_-(\la,x),\quad \la\in\C,\ \ x\in \R.
 \eeq
 Moreover,
\beq\label{vazhn8}
\phi_\pm(\la,x)\ol{\theta_\pm(\la,x)}=\ol{\phi_\pm(\la,x)}\theta_\pm(\la,x),\quad  \la\in \Sigma^{(2)}.
\eeq
It remains to show that $\phi_\pm(\la,x)=\theta_\pm(\la,x)$, or equivalently, that for all $\la\in\C$ and $x\in\R$
\[
 p(\la,x):=\frac{\phi_-(\la,x)}{\theta_-(\la,x)}=\frac{\theta_+(\la,x)}{\phi_+(\la,x)}\equiv 1.
\]
We proceed as in \cite{M}, Section 3.5, or as in in \cite{bet}, Section 5.
We first exclude from our consideration the discrete set $\mathcal O$ of parameters $x\in\R$ for which at least one of the following equalities
 is fulfilled: $\phi(E,x)=0$ for $E\in\{\la_1,\dots,\la_p, c_-, c_+\}$. We begin by showing that for each $x\notin \mathcal O$ the equality $\phi_+(\hat\la,x)=0$ implies the equality $\theta_+(\hat\la,x)=0$. Indeed, since $\hat\la\notin\{\la_1,\dots,\la_p, c_-, c_+\}$ we have $W(\hat\la)\neq 0$ and therefore by (iv) of Lemma \ref{3.3} that $\theta_-(\hat\la,x)\neq 0$. But then from \eqref{vazhn10} the equality $\theta_+(\hat\la,x)=0$
  follows. Thus the function $p(\la,x)$ is holomorphic in $\mathcal D$. By (ii) of Lemma \ref{3.3}, it has no jump along the set $(c_-, c_+)$, and by \eqref{vazhn8} it has no jump along $\la\geq c_+$. Since $\phi_+(c_\pm,x)\neq 0$, then \eqref{vazhn} and \eqref{vazhn2} imply that $p(\la,x)$ has removable singularities at $c_+$ and $c_-$. By (iii) of Lemma \ref{3.3} 
 $p(\la)\to 1$ as $\la\to\infty$, and by Liouville's theorem $p(\la,x)\equiv 1$ for $x\notin\mathcal O$. But the set $\mathcal O$ is discrete, therefore, by continuity  $\phi_\pm(\la,x)=\theta_\pm(\la,x)$ for all $\la\in\C$ and $x\in\R$. In turn this implies that $q_-(x)=q_+(x)$ and
completes the proof of Theorem \ref{theor2}. 
\end{proof}

\section{Additional properties of the scattering data}

In this section we study the behavior of the reflection coefficients as $\la\to\infty$ and its connection to the smoothness of the potential. One should emphasize that the rough estimate {\bf I.\ (e)} is sufficient for solving the inverse scattering problem (independent of the number of derivatives $n$), because this information is contained in property {\bf IV} of the Fourier transforms of the reflection coefficients. That is why we did not include the estimate from Theorem \ref{reflcomp} proved below in the list of necessary and sufficient conditions. On the other hand, this estimate plays an important role in application of the IST for solving the Cauchy problem for KdV equation with steplike initial profile. Lemma \ref{Weyllem} and Theorem \ref{reflcomp} clarify and improve corresponding results of \cite{bet} and are of independent interest for the spectral analysis of $L$.

We introduce the following notation: We will say that a function $g(\la)$, defined on the set $\mathcal A:=\Sigma\cap  \{\la\geq a\gg\ol c\}$, belongs to the space $L^2(\infty)$ if it satisfies the symmetry property $g(\la+\I 0)=\ol{g(\la - \I 0)}$ on $\mathcal A$, and
\[
\int_a^{+\infty}|g(\la)|^2\frac{d\la}{|\sqrt\la|}<\infty.
\]
Note that this definition implies $g(\la)\in L^2_{\{k_\pm\}}(\R\setminus (-a,a))$ for sufficiently large $a$.

\begin{theorem}\label{reflcomp}
Let $q\in \mathcal L_m^n(c_+, c_-)$, $m,n\geq 1$. Then for $\la\to\infty$
\[
\frac{d^s}{d k_\pm^s} R_\pm(\la)=g_{\pm,s}(\la)\,\la^{-\frac{n+1}{2}},\quad s=0,1,\dots, m-1,
\]
where $g_{\pm,s}(\la)\in L^2(\infty)$.
\end{theorem}

Note that the case $n=0$ and $m=1$ already follows from Lemma \ref{rem7}, since (using the notation of its proof) $R_\pm(\la)=f_{6,\pm}k_\pm^{-1}$ admits $m-1$ derivatives with respect to $k_\pm$ for $m>1$, and $f_{6,\pm}^{(s)}\in  L^2_{\{k_\pm\}}(\R\setminus (-a,a))$. The general case will be
shown at the end of this section.
Using Lemma \ref{Best} and formula \eqref{der4}, we can specify an asymptotical expansion for the Jost solution of equation \eqref{Sp} with a smooth potential.

\begin{lemma}\label{expansion}
Let  $q\in\mathcal L_m^n(c_+, c_-)$ and $q_\pm(x)=q(x) - c_\pm$. Then for large $k_\pm\in\R$ the  Jost solution $\phi_\pm(\la,x)$ of the equation $L\phi_\pm=\la\phi_\pm$ admits an asymptotical expansion
\beq\label{exp1}
\phi_\pm(\la,x)=\E^{\pm\I k_\pm x}\left( u_{\pm,0}(x) \pm \frac{u_{\pm,1}(x)}{ 2\I k_\pm} +\dots+
\frac{u_{\pm,n}(x)}{(\pm 2\I k_\pm)^n} + \frac{U_{\pm,n}(\la,x)}{(\pm 2\I k_\pm)^{n+1}}\right),
\eeq
where
\beq\label{uu5}
u_0(x)=1,\quad u_{\pm,l+1}(x)=\int_x^{\pm\infty} (u_{\pm,l}^{\prime\prime}(\xi) - q_\pm(\xi)u_{\pm,l}(\xi))d\xi,\quad l=1,\dots,n.
\eeq
Moreover, the functions $U_{\pm,n}(\la,x)$ and $\frac{\pa}{\pa x}U_{\pm,n}(\la,x)$ are $m-1$ times differentiable with respect to $k_\pm$ with the following behavior as $\la\to\infty$ and $ 0\leq s\leq m-1 $:
\beq\label{2}
\frac{\partial^s}{\partial k_\pm^s} U_{\pm,n}(\la, x)\in L^2(\infty),\quad
\frac{\partial^s}{\partial k_\pm^s} \left(\frac{1}{k_\pm}\frac{\partial}{\partial x}U_{\pm,n}(\la, x)\right)\in L^2(\infty).
\eeq
\end{lemma}

\begin{proof}
Formula \eqref{der4} implies
\beq\label{der4l}
\frac{\partial^s B_\pm(x,y)}{\partial y^s}=\frac{\partial^s B_\pm(x,y)}{\partial x\partial y^{s-1}} +\int_x^{\pm\infty}q_\pm(\alpha) \frac{\partial^{s-1} B_\pm(\alpha,y)}{\partial y^{s-1}} d\alpha,\quad s\geq 1.
\eeq
Integrating \eqref{Jost11} by parts and taking into account \eqref{der4l} with $s=n+1$ and  Lemma \ref{Best}, we get
\begin{align} \nn
\phi_\pm(k_\pm,x)\E^{\mp\I k_\pm x}= &1 \mp \frac{1}{2\I k_\pm} B_\pm(x,0) \pm\cdots+\frac{(-1)^{n}}{(\pm 2\I k_\pm)^n}\frac{\pa^{n-1}B_\pm(x,0)}{\pa y^{n-1}} \\ \nn
+ &\frac{(-1)^{n+1}}{(\pm 2\I k)^{n+1}}\left\{\frac{\pa^{n}B_\pm(x,0)}
{\pa y^{n}}\pm \int_0^{\pm\infty}\left(\frac{\pa}{\pa x}
\frac{\pa^{n}}{\pa y^{n}}B_\pm(x,y) \right.\right.\\ \label{uu7}
+&\left.\left.\int_x^{\pm\infty}q_\pm(\alpha)\frac{\pa^{n}}{\pa y^{n}}B_\pm(\alpha, y)d\alpha\right)\E^{\pm 2\I k_\pm y} dy\right\}.
\end{align}
Set
\[
u_{\pm,l}(x):=(-1)^{l} \frac{\pa^{l-1}B_\pm(x,0)}{\pa y^{l-1}},\ \ \ l\leq n+1.
\]
Then \eqref{der4l} implies  \eqref{uu5}. Put
\beq\label{defu}
u_{\pm,l+1}(x,y)=(-1)^{l+1}\frac{\pa^{l}B_\pm(x,y)}{\pa y^l},\ \ l\leq n.
\eeq 
By \eqref{decay}, \eqref{decay1}, \eqref{defsig}, \eqref{defsig1}, and \eqref{Best1} we have $\nu_{\pm,l}(\cdot)\in \mathcal L_{m-1}^0(\R_\pm)$. This implies
\beq\label{uu6}
u_{\pm,n+1}(x,\cdot),\ \frac{\partial}{\partial x}u_{\pm,n+1}(x,\cdot)\in \mathcal L^0_{m-1}(\R_\pm).
\eeq
Comparing \eqref{exp1} with \eqref{uu7} gives
\begin{align} \label{UU}
U_{\pm, n}(\la,x)&=u_{\pm, n+1}(x)
+ \int_0^{\pm\infty} \left(\frac{\partial}{\partial x} u_{\pm,n+1}(x,y)\right.\\ \nn
\pm&\left. \int_x^{\pm\infty} q_\pm(\alpha) u_{\pm,n+1}(\alpha,y)d\alpha\right) 
\E^{\pm 2\I k_\pm y} dy,
\end{align}
where the function $u_{\pm,n+1}(x,y)$, defined by \eqref{defu}, satisfies $u_{\pm,n+1}(x,0)=u_{\pm,n+1}(x)$.
From \eqref{uu5}, it follows that the representation for $u_{l,\pm}(x)$ involves $q_\pm^{(l-2)}(x)$ and  lower order derivatives of the potential. Thus $u_{\pm,n+1}(x)$ can be differentiated only one more time with respect to $x$.
But we cannot differentiate the right-hand side of \eqref{UU} directly under the integral. To avoid this, let us first integrate by parts the first summand in this integral. By \eqref{defu},  we have $\frac{\partial }{\partial y}u_{\pm,n}(x,y)=-u_{\pm,n+1}(x,y)$. Taking the derivative with respect to $x$ outside the integral, we get
\[
\int_0^{\pm\infty}\frac{\pa}{\pa x} u_{\pm,n+1}(x,y)\E^{\pm 2\I k_\pm y} dy=\frac{d}{dx}\left(u_{\pm,n}(x) \mp 2\I k_\pm\int_0^{\pm\infty} u_{\pm,n}(x,y)\E^{\pm 2\I k_\pm y} d y\right).
\]
According to \eqref{uu5}, we have
$u_{\pm,n+1}^\prime(x) +u_{\pm,n}^{\prime\prime}(x)=q_\pm(x)u_{\pm,n}(x)$,
and therefore
\[
\frac{\pa}{\pa x} U_{\pm,n}(\la,x)=2\I k_\pm\left(\frac{q_\pm(x)u_{\pm,n}(x)}{(2\I k_\pm)}
\mp\int_0^{\pm\infty}  \frac{\pa}{\pa x} u_{\pm,n}(x,y)\E^{\pm 2\I k_\pm y} d y\right) -
\]
\[
 \mp\int_0^{\pm\infty}u_{\pm,n+1}(x,y)q_\pm(x)\E^{\pm 2\I k_\pm y} d y,
\]
which together with \eqref{uu6} proves  \eqref{2}.
\end{proof}

Our next step is to specify an asymptotic expansion for the  Weyl functions
\beq\label{Weyl1}
m_\pm(\la,x)=\frac{\phi_\pm^\prime(\la,x)}{\phi_\pm(\la,x)}
\eeq
for the Schr\"odinger equation. Note that due to estimate \eqref{zeros} and continuity of $\hat\sigma(x)$ for any $b>0$ there exist some $k_0>0$ such that for all real $k_\pm$ with $|k_\pm|>k_0$ the function $\phi_\pm(\la,x)$ does not have zeros for $|x|<b$. Therefore, $m_\pm(k_\pm,x)$ is well-defined for all large real $k_\pm$ and $x$ in any compact set $\mathcal K\subset\R$.

\begin{lemma}\label{Weyllem}
Let  $q\in\mathcal L_m^n(c_+, c_-)$.
Then for large $\la\in\R_+$ the  Weyl functions \eqref{Weyl1} admit the asymptotic expansion
\beq\label{Weyl5}
m_\pm(k,x)= \pm\I\sqrt\la + \sum_{j=1}^n \frac{m_j(x)}{(\pm 2\I\sqrt\la)^j} +\frac{m_{\pm,n}(\la,x)}{(\pm2\I \sqrt\la)^n},
\eeq
where
\beq\label{rec}
m_1(x)=q(x),\ \ m_{l+1}(x) =-\frac{d}{dx} m_l(x) -\sum_{j=1}^{l-1}m_{l-j}(x)m_{j}(x),
\eeq
and the functions $m_{\pm,n}(\la,x)$ are $m-1$ times differentiable with respect to $k_\pm$ with
\beq\label{WW7}
\frac{\pa^s}{\pa k_\pm^s} m_n(\la,x)\in L^2(\infty),\quad s\leq m-1,\ \forall x \in\mathcal K.
\eeq
\end{lemma}

\begin{remark}
The recurrence relations \eqref{rec} are well-known for the case of the Schr\"odinger operator with smooth potentials and are usually proven via the Riccati equation satisfied by the Weyl functions.
Our point here is the fact that \eqref{Weyl5} is $m-1$ times differentiable with respect to $k_\pm$ together with \eqref{WW7}.
\end{remark}

\begin{proof}
We follow the proof of \cite{M}, Lemma 1.4.2, adapting it for the steplike case.
From \eqref{Weyl1} and \eqref{Sp}, we have $m_\pm(\la,x)=\I k_\pm + \kappa_\pm(\la,x)$, where $\kappa_\pm(\la,x)$
satisfy the equations
\[
\kappa_\pm^\prime(\la,x) \pm 2\I k_\pm \kappa_\pm(\la,x) + \kappa_\pm^2(\la,x) - q_\pm(x)=0,\quad \kappa_\pm(\la, x)=o(1),\quad \la\to\infty.
\]
Introduce notations $\phi_\pm(\la,x)=\E^{\pm\I k_\pm x} Q_{\pm,n}(\la,x)$, where (cf. Lemma \ref{expansion})
\begin{align}\label{Qn}
Q_{\pm,n}(\la,x) &:=P_{\pm,n}(\la,x)+ \frac{U_{\pm,n}(\la,x)}{(\pm 2\I k_\pm)^{n+1}},\\ \label{Pn}
P_{\pm,n}(\la,x) &:=1 +\frac{u_{\pm,1}(x)}{(\pm 2\I k)} +\cdots+ \frac{u_{\pm,n}(x)}{(\pm 2\I k)^n}.
\end{align}
Then
\[
\kappa_\pm(\la,x)=\frac{ P_{\pm,n}^\prime(\la,x)}{P_{\pm, n}(\la,x)} +\frac{ U_{\pm,n}^\prime(\la,x) P_{\pm,n}(\la,x) -U_{\pm,n}(\la,x)
P_{\pm,n}^\prime(\la,x)}{(\pm 2\I k_\pm)^{n+1} P_{\pm,n}(\la,x) Q_{\pm,n}(\la,x)}.
\]
Decompose the first fraction in a series with respect to $( 2\I k_\pm)^{-1}$ using \eqref{Pn}. Since $P_{\pm,n}(\la,x)\neq 0$  for $x\in\mathcal K$ and sufficiently large $\la$, then we get
\beq\label{seri}
\frac{ P_{\pm,n}^\prime(\la,x)}{P_{\pm,n}(\la,x)} =\sum_{j=1}^n \frac{\kappa_{\pm,j}(x)}{(\pm 2\I k_\pm)^j}  + \frac{f_{\pm,n}(\la,x)}{(\pm 2\I k_\pm)^n},\eeq
where $\kappa_{\pm,j}(x)$ are polynomials of $u_{\pm,l}$, $l\leq j$, and
 the function $f_{\pm,n}(\la,x)$ is infinitely many times differentiable with respect to $k_\pm$ for sufficiently big $k_\pm$, and 
\beq\label{f}
\frac{\pa^l}{\pa k_\pm^{l}}f(\la,x)\in L^2(\infty),\ l=0,1,\dots
\eeq
Correspondingly,
\beq\label{seri2}
\kappa_\pm(\la,x)=\sum_{j=1}^n \frac{\kappa_{\pm,j}(x)}{(\pm2\I k_\pm)^j} +\frac{\kappa_{\pm,n}(\la,x)}{(2\I k_\pm)^n},
\eeq
where
\[
\kappa_{\pm,n}(\la,x)=f_{\pm,n}(k,x) +
\frac{ U_{\pm,n}^\prime(\la,x)}{2\I k_\pm  Q_{\pm,n}(\la,x)} - \frac{ U_{\pm,n}(\la,x) P_{\pm,n}^\prime(\la,x)}{2\I k_\pm P_{\pm,n}(\la,x) Q_{\pm,n}(\la,x)}.
\]
Taking into account \eqref{uu5}, \eqref{uu6}, \eqref{2},
\eqref{Qn}, \eqref{Pn}, and \eqref{f}, we get 
\[
\frac{\pa^s}{\pa k_\pm^s} \kappa_{\pm,n}(\la,x)\in L^2(\infty),\quad s\leq m-1,\ \forall x \in\mathcal K.
\]
Next, due to \eqref{uu5}, the functions $u_l(x)$ depend on $q^{(l-2)}(x)$ and lower order derivatives of the potential, and can be differentiated at least twice more with respect to $x$ for $l\leq n$.
Since  the function $\phi_{\pm}(\la,x)$ itself is also twice differentiable with respect to $x$, the same is valid for
$U_{\pm,n}(\la,x)$ and $\kappa_\pm(\la,x)$. Hence each summand of \eqref{Pn} can be differentiated twice, and we
conclude that all $\kappa_{\pm,j}(x)$, $j\leq n$, in \eqref{seri2} are differentiable with respect to $x$, and so is $\kappa_{\pm,n}(\la,x)$.

Next, for large $\la$, we can expand $k_\pm$ with respect to $\sqrt\la$ and represent $m_\pm(\la,x)$ using \eqref{seri2} as
$m_\pm(\la,x)=\pm\I\sqrt\la +\tilde\kappa_\pm(\la,x)$, where
\[
\tilde\kappa_\pm(\la,x)=\sum_{j=1}^n \frac{\tilde\kappa_{\pm,j}(x)}{(\pm 2\I \sqrt\la)^j} +\frac{m_{\pm,n}(\la,x)}{(2\I \sqrt\la)^n}.
\]
Here $\tilde\kappa_{\pm,j}(x)$ are some other coefficients, but they also depend on the potential and its derivatives up to order $n-1$, i.e.,\ one time differentiable together with
$\tilde\kappa_{\pm,n}(\la,x)$ with respect to $x$. Moreover, $m_{\pm,n}(\la,x)$ satisfies the same estimates as in
\eqref{WW7}.
But  $\tilde\kappa_\pm(\la,x)$ satisfies the Riccati equation
\[
\tilde\kappa_\pm^\prime(\la,x) \pm 2\I \sqrt\la \kappa_\pm(\la,x) + \kappa_\pm^2(\la,x) - q(x)=0,
\]
and therefore $\tilde\kappa_{+,l}(x)=\tilde\kappa_{-,l}(x)=m_l(x)$, where $m_l(x)$ satisfies \eqref{rec}. 
\end{proof}

\begin{corollary}\label{lemim}
Let $q\in \mathcal L_m^n(c_+,c_-)$ with $n\geq 1$ and $m\geq 1$. Then for any $\mathcal K\subset\R$, $x\in\mathcal K$ and sufficiently large  $\la>\ol c$
the  function
\[
f_{\pm,n}(\la,x):=k_\pm^n\left(\overline{m_\pm(\la,x)} -m_\mp(\la,x)\right)
\]
is $m-1$ times differentiable with respect to $k_\pm$ with
\[
\frac{\pa^s}{\pa k_\pm^s} f_{\pm,n}(\la,x)\in L^2(\infty),\quad 0\leq s\leq m-1.
\]
\end{corollary}

The claim of Theorem \ref{reflcomp} follows immediately from \eqref{2.17}, evaluated for $x\in\mathcal K$, \eqref{refff}, \eqref{Weyl1}, Lemma \ref{Weyllem}, and Corollary \ref{lemim}.

\bigskip
\noindent{\bf Acknowledgment.}

I.E.\ gratefully acknowledges the hospitality of the Department of Mathematics of Purdue University where this work was initiated.

\end{document}